\documentclass{article}

\usepackage{hyperref}
\usepackage{blkarray}

\usepackage{amsmath}
\usepackage{amssymb}
\usepackage{amsfonts}
\usepackage{amsthm}
\usepackage{enumerate}
\usepackage{multicol}
\usepackage{colortbl}

\usepackage[final]{epsfig}

\usepackage[english]{babel}
\usepackage{tabularx}
\usepackage{subfigure}

\newcommand{\rank}{{\rm rank}}
\newcommand{\p}{p}%{{\bf p}}
 
\newcommand{\bu}{{u}}
\newcommand{\bbm}{m}%{{\bf m}} 

\newcommand{\C}{\mathcal C}

\newcommand{\gs}{\mathcal G_{\C}} %gain space

\newcommand{\pogp}{\langle G', \bbm' \rangle}

\newcommand{\Tor}{\mathcal T_0} %Fixed Torus
\newcommand{\pofw}{(\langle G, \bbm \rangle,\p)} %Periodic Orbit FrameWork
\newcommand{\pog}{\langle G, \bbm \rangle} %Periodic Orbit Graph
\newcommand{\dpfw}{( G^{\bbm}, \p^{\bbm})} %Derived Periodic FrameWork
\newcommand{\bbofw}{(\langle H, \bbm \rangle,q)} %Bar Body Orbit FrameWork
\newcommand{\bbog}{\langle H, \bbm \rangle} %Bar Body Orbit Graph
\newcommand{\pogind}{\langle G_H, \bbm_H \rangle} %Induced Periodic Orbit Graph
\newcommand{\pofwind}{(\langle G_H, \bbm_H \rangle, p_H)} %Induced Periodic Orbit Framework
\newcommand{\bbogp}{\langle H', \bbm' \rangle} %Bar-Body Orbit Graph Prime

\newcommand{\R}{\mathbf R} %Rigidity Matrix

\newcommand{\bbgood}{$[6,3]$-$\Tor^3$-graph}

\definecolor{desk}{rgb}{.345, .306, .216}
\definecolor{vancouver}{rgb}{.412, .412,.412}
\definecolor{beetle}{rgb}{.180, .161, .102}
\definecolor{bluey}{rgb}{.235, .380, .415}
\definecolor{melon}{rgb}{1, .259, .259}
\definecolor{vneck}{rgb}{.596, .282, .376}
\definecolor{pink}{rgb}{.918, .122, .545}
\definecolor{mango}{rgb}{1, .8, .267}
\definecolor{lips}{rgb}{.541, .074, .239}
\definecolor{sage}{rgb}{.522, .604, .247}
\definecolor{moss}{rgb}{.184, .224, .129}
\definecolor{cumin}{rgb}{.6, .580, 0}
\definecolor{lichen}{rgb}{.745, .998, .729}
\definecolor{rain}{rgb}{.780, .812, .706 }
\definecolor{cloud}{rgb}{.961, .976, .870}
\definecolor{couch}{rgb}{.8, 1, .2}
\definecolor{cement}{rgb}{.678, .682, .549}
\definecolor{sky}{rgb}{.278, .514, 1}

\newtheorem{thm}{Theorem}[section]

\newtheorem{lem}[thm]{Lemma}
\newtheorem{prop}[thm]{Proposition}
\newtheorem{conj}[thm]{Conjecture}

\newtheorem*{thma}{Theorem}

\theoremstyle{remark}
\newtheorem{rem}[thm]{Remark}
\theoremstyle{remark}
\newtheorem{ex}[thm]{Example}
\theoremstyle{remark}

\begin{document}

\title{The rigidity of periodic body-bar frameworks on the three-dimensional fixed torus}
\author{
{Elissa  Ross
\thanks{Department of Mathematics and Statistics, York University, Toronto, Canada.  elissa@mathstat.yorku.ca}}
}
\date{\today}
\maketitle

\begin{abstract} 
We present necessary and sufficient conditions for the generic rigidity of body-bar frameworks on the three-dimensional fixed torus. These frameworks correspond to infinite periodic body-bar frameworks in $\mathbb R^3$ with a fixed periodic lattice. \\	

%rigidity and flexibility of structures

\end{abstract}

%\tableofcontents

%INTRODUCTION
\section{Introduction}

The study of periodic structures from the perspective of rigidity theory has received considerable attention in recent years, due in part to questions raised about the material properties of zeolites. Zeolites are a type of micro-porous mineral whose molecular structure is periodic in nature. Since the properties of zeolites are related to their structural properties, it has been a topic of interest to determine theoretically when such a material is rigid or flexible \cite{flexibilityWindow}. 

Toward this end, there has been a surge of interest in the study of {\it periodic frameworks}. A periodic framework is composed of rigid bars linked together periodically to form an infinite repetitive bar-joint framework, that is, rigid bars connected by flexible joints. This can be described by an infinite graph $\tilde G$, together with a group, say $\Gamma$, which describes the periodicity of the graph $\tilde G$. Together with a periodic placement $\tilde p$ of the vertices of $\tilde G$ in $\mathbb R^d$, we have a periodic framework $(\tilde G, \tilde p)$ \cite{BS1}. The periodicity of the framework determines the {\it periodic lattice}, which we may consider as either fixed or variable. The periodic lattice is generated by the $d$ translations under which $(\tilde G, \tilde p)$ is invariant.

In the thesis of the author, a combinatorial characterization for the {\it generic rigidity} of two-dimensional periodic frameworks with a fixed periodic lattice was presented. That is, we found combinatorial conditions on a {\it periodic orbit graph} which are sufficient to guarantee the rigidity of {\it almost all} realizations of the graph as periodic frameworks. 
Subsequent work by Malestein and Theran \cite{theran} characterized two dimensional periodic frameworks with a variable periodic lattice. Unfortunately, like the study of the rigidity of finite graphs, we lack combinatorial tools to predict the generic rigidity of periodic frameworks for $d >2$. 

It is natural, then, that we turn our attention to {\it periodic body-bar frameworks}, which are formed of rigid bodies linked periodically together by bars (two bodies can be linked by multiple bars, hence body-bar frameworks are captured by multigraphs, in contrast to bar-joint frameworks which are always described by simple graphs). In their finite incarnations, body-bar frameworks admit neat combinatorial characterizations for all $d$ \cite{Tay2}. The question then becomes, do periodic body-bar frameworks also have a nice combinatorial characterization? 

In \cite{myThesis}, a complete characterization of periodic body-bar frameworks with a fixed lattice was conjectured by the author. In this paper, we prove the result for $d=3$ (with $d = 1,2$ following from the previous bar-joint characterizations \cite{myThesis}). In particular, we show:
\begin{thma}
$\bbog$ is a periodic orbit graph corresponding to a generically minimally rigid body-bar periodic framework in $\mathbb R^3$ if and only if 
\begin{enumerate}
	\item $|E(H)| = 6|V(H)| - 3$
	\item for all non-empty subsets $Y \subset E(H)$ of edges
		\begin{equation*}
			|Y| \leq 6|V(Y)| - 6 + \sum_{i=1}^{|\gs(Y)|}(3-i).
		\end{equation*}	
\end{enumerate}
\end{thma}
The sparsity condition described by condition {\it 2} depends on the dimension of the {\it gain space}  $\gs(Y)$ of a set of edges $Y$, which, roughly speaking, can be thought of as part of the connectivity information about the periodic framework. Alternatively, $\gs(Y)$ describes the way in which the set of edges $Y$ ``wrap" around a torus (as a model of periodic space), and can thus be seen as information about the homotopy type of the edge set $Y$. 

We mention some very recent related work due to Borcea, Streinu and Tanigawa \cite{PeriodicBarBody}.  In that paper, the authors find necessary and sufficient combinatorial conditions for rigidity on {\it unlabelled} quotient graphs of infinite periodic body-bar frameworks. The work in the present paper is concerned with {\it labelled} quotient graphs of infinite periodic frameworks. Given a particular infinite periodic body-bar framework, this defines a unique (up to size of unit cell and translation) labeled quotient graph, which we will call a {\it periodic orbit framework}. The analysis of labelled quotient graphs is a harder problem than the analysis of unlabelled ones.  The approach of \cite{PeriodicBarBody} will only tell us that a good labelling (``lifting") of the quotient graph exists, but cannot predict whether a particular labelling will be good. That said, the paper \cite{PeriodicBarBody} considers $d$-dimensional frameworks, and does so with a variable periodic lattice. That is, strictly speaking, a harder problem to characterize than the fixed lattice considered here.  To further differentiate the present work from the recent paper \cite{PeriodicBarBody}, the methods used here are inductive methods, while the approach of Borcea, Streinu and Tanigawa is non-inductive. 

We conclude this brief introduction by mentioning some reasons why it may be valuable to consider periodic frameworks with a fixed periodic lattice. It may, at first glance, seem like an unnatural restriction or simplification of the problem of periodic frameworks (with a variable lattice). However, for frameworks with a variable periodic lattice which are flexible, the vertices which are ``far away" from the ``centre" (arbitrarily chosen) are moving arbitrarily quickly. Some materials scientists have suggested that the time scales of lattice movement are several orders of magnitude slower than molecular deformation within the lattice \cite{ThorpePrivate}. Finally, we may also view the fixed lattice as a stepping stone or preview of a full characterization of frameworks with a variable lattice. In particular, it is possible to view characterizations of frameworks with a variable lattice as consisting of two parts: a component which ``rigidifies" the variability of the lattice, and a second component which is rigid within the resulting fixed lattice. 

\subsection{Outline of paper}
We provide an extensive background section (Section \ref{sec:background}) which outlines several types of frameworks, and their rigidity. In particular, we discuss finite bar-joint frameworks, bar-joint periodic orbit frameworks and their corresponding derived periodic (bar-joint) frameworks, finite body-bar frameworks, body-bar periodic orbit frameworks and their corresponding derived periodic (body-bar) frameworks. We also describe induced bar-joint frameworks from body-bar frameworks. Section \ref{sec:barJointInductions} provides information about rigidity-preserving inductive constructions for three-dimensional bar-joint periodic orbit frameworks. In Section \ref{sec:mainResult} we state and motivate our main result, and establish the necessity of the characterization. Section \ref{sec:sufficient} is the main work of the paper, and details the proof of the sufficiency of the characterization. We conclude with some conjectures and ideas for further work. 

%BACKGROUND
\section{Background}
\label{sec:background}
In this section we outline the basic notions of finite and periodic bar-joint frameworks, and finite and periodic body-bar frameworks. 

%FINITE BAR-JOINT
\subsection{Finite bar-joint frameworks}
\label{sec:finiteBarJoint}
A {\it bar-joint framework} $(G, p)$ is a finite simple graph $G$ together with a position of the vertices of $G$ in Euclidean space, $p:V(G) \rightarrow \mathbb R^d$, with $p(v_i) \neq p(v_j)$ for all $\{v_i, v_j\} \in E(G)$. We write $p(v_i) = p_i$, $p_i \in \mathbb R^d$. An {\it infinitesimal motion} of $(G, p)$ is a function $u: V(G) \rightarrow \mathbb R^d$ (we write $u(v_i) = u_i$) such that 
\[(u_i - u_j) \cdot (p_i - p_j) = 0 \ \textrm{for all edges} \ \ e = \{v_i, v_j\} \in E(G).\]
In other words, an infinitesimal motion instantaneously preserves the lengths of the bars of the framework. 
An infinitesimal motion of $(G, p)$ is {\it trivial} if it corresponds to an infinitesimal rigid motion (e.g. a rotation or translation). If the only infinitesimal motions of $(G, p)$ are trivial, we say that $(G,p)$ is {\it infinitesimally rigid}, otherwise $(G,p)$ is {\it infinitesimally flexible}. 

It is clear that if the graph $G$ is disconnected, then $(G, p)$ is always infinitesimally flexible, hence we assume throughout that $G$ is connected. 

In this paper we focus our attention on infinitesimal rigidity, which is sufficient to guarantee the continuous rigidity of a framework in $\mathbb R^d$, although the converse is false. The focus on infinitesimal rigidity is thus a typical way to attack the problem of continuous rigidity, which is computationally challenging. See any basic reference on rigidity theory for details \cite{CombRigidity,CountingFrameworks,SomeMatroids,SceneAnalysis}.

The vector space of infinitesimal motions of a framework $(G, p)$ is given by the kernel of the {\it rigidity matrix}, $\R(G,p)$. This is the $|E| \times d|V|$ matrix with $d$ columns for each vertex, and one row corresponding to each edge $e = \{v_i, v_j\} \in E(G)$, as follows: 
\[\left(\begin{array}{ccccc}0 \dots 0 & \p_i - \p_j & 0 \dots 0 & \p_j-\p_i & 0\dots 0\end{array}\right),\]
where the non-zero entries occur in the columns corresponding to vertices $v_i$ and $v_j$ respectively. 
Note that there will always be a ${d+1 \choose 2}$-dimensional space of trivial solutions of $\R(G, p)$, corresponding to infinitesimal rigid motions of $\mathbb R^d$ (translations, rotations). The following basic theorem follows: 
\begin{thm}[see for example \cite{SceneAnalysis}]
A framework $(G, p)$ with $|V(G)|>d$ is infinitesimally rigid in $\mathbb R^d$ if and only if 
\[\rank \R(G, p) = d|V(G)| - {d+1 \choose 2}.\]
\label{thm:finiteMatrix}
\end{thm}
If $(G,p)$ is infinitesimally rigid and satisfies $|E(G)| = d|V(G)| - {d+1 \choose 2}$, we say that $(G,p)$ is {\it isostatic}, meaning that it is minimally rigid. The notion of minimal rigidity also corresponds to the idea of a framework which is {\it independent}, meaning that the rows of the rigidity matrix are linearly independent. We elaborate on this idea in the subsequent section. 

We remark finally that the rigidity of frameworks is a generic property: for almost all configurations $p$ of a graph $G$, the framework $(G,p)$ will be either infinitesimally rigid or infinitesimally flexible. In particular, we say that a {\it generic} position $p$ of the graph $G$ is one such that all minors of $\R(G, p)$ that have determinants which are not identically zero have non-zero determinant. This yields an open dense set of generic positions for a given graph $G$.  We say that a graph $G$ is {\it generically rigid} if it is infinitesimally rigid at all generic positions $p$. There are other ways to define generic, for example we may demand that the coordinates of the vertices of $G$ be algebraically independent over the rationals. In this case, the set of generic configurations is dense but not open.

The preceding definitions and notions are entirely standard. For further details see for example \cite{CombRigidity,CountingFrameworks,SomeMatroids} or \cite{SceneAnalysis}. 

%BAR JOINT ORBIT FRAMEWORKS
\subsection{Bar-joint periodic orbit frameworks on the fixed torus $\Tor^d$}
\label{sec:barJointOrbit}
To describe periodic frameworks, we use the language of gain graphs (also known as {\it voltage graphs} \cite{TopologicalGraphTheory}). A {\it gain graph} is a finite multigraph $G$ whose edges are labeled invertibly by the elements of a group. Suppose the edges of $G$ are labeled by elements of $\mathbb Z^d$, with the function $m:E(G)^+ \rightarrow \mathbb Z^d$. We say that the pair $\pog$ is a (bar-joint) {\it periodic orbit graph}, for reasons that will soon become clear (Figure \ref{fig:gainGraph} (a)). The vertices of $\pog$ are simply the vertices of the graph $G$. The edges of $\pog$ are recorded $e= \{v_i, v_j; m_e\}$, where $v_i, v_j \in V(G),  m_e \in \mathbb Z^d$. Since the edges are labeled invertibly by elements of $\mathbb Z^d$, it follows that the edge $e$ can equivalently be written: 
\[e= \{v_i, v_j; m_e\} = \{v_j, v_i; -m_e\}.\]

From the periodic orbit graph $\pog$, we may define the {\it derived periodic graph} $G^m$ which is the (infinite) graph whose vertex and edge sets are given by $V(G) \times \mathbb Z^d$ and $E(G) \times \mathbb Z^d$ respectively ( Figure \ref{fig:gainGraph}(b)) . If $v_i$ is a vertex of $\pog$, we say that $(v_i, z), z \in \mathbb Z^d$ is the {\it orbit} of $v_i$ in $G^m$.  Similarly, if $e = \{v_i, v_j; m_e\}$ is an edge of $\pog$, then the {\it orbit} of edges in $G^m$ corresponding to $e$ is given by:
\[\{(v_i, z_1), (v_j, z_2+m_e)\}, \ \textrm{where} \ v_i, v_j \in V(G), z_1, z_2, m_e \in \mathbb Z^d.\]

\begin{figure}[h!]
\begin{center}
\subfigure[$\pog$]{\label{fig:periodicFramework}\includegraphics[width=1.5in]{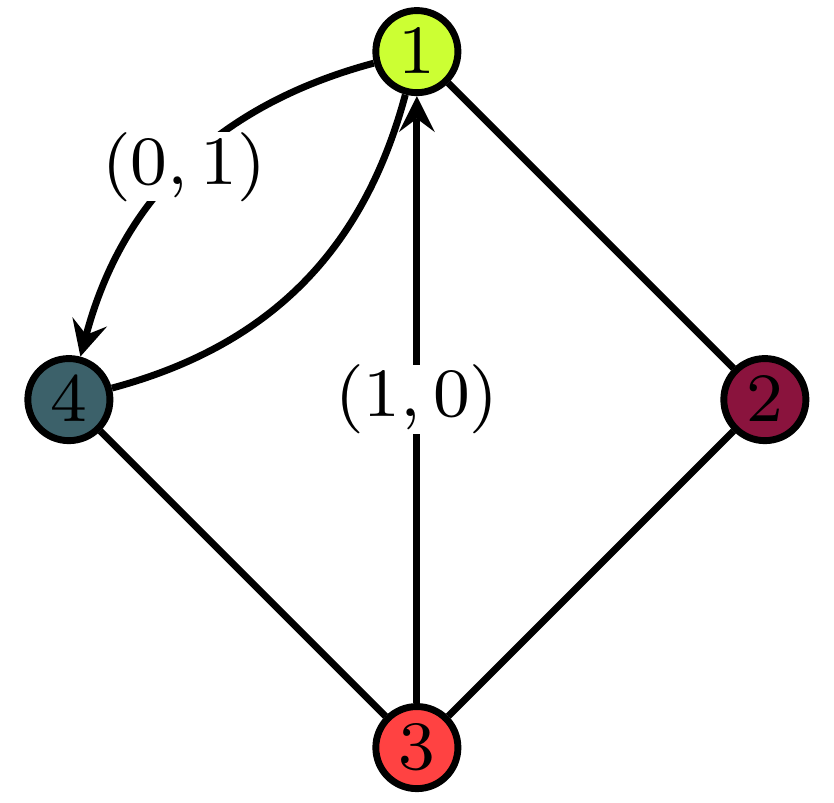}}
\hspace{.5in}%
\subfigure[$G^m$]{\label{fig:periodicFramework}\includegraphics[width=1.5in]{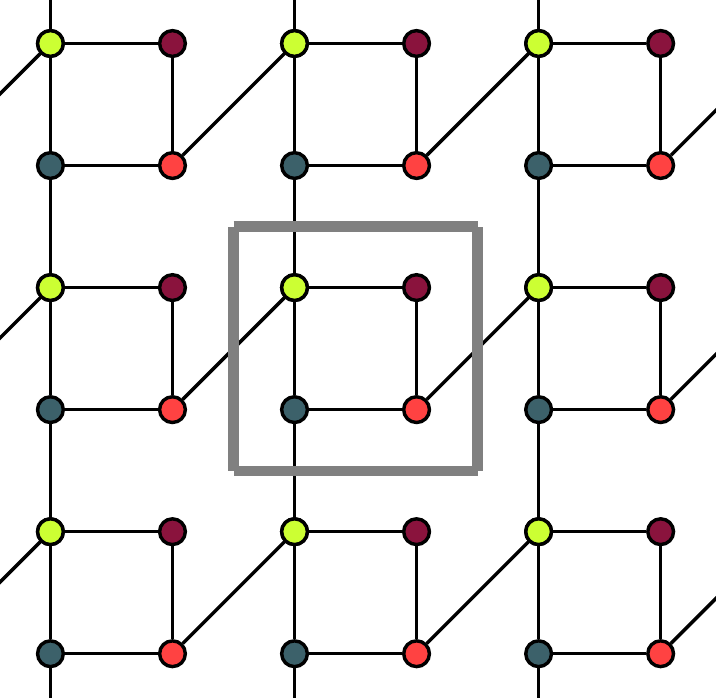}}
\caption{A gain graph $\pog$, where $m:E \rightarrow \mathbb Z^2$, and its derived graph $G^m$.   \label{fig:gainGraph}}
\end{center}
\end{figure}

Throughout this paper, we are concerned with periodic frameworks with a fixed periodic lattice, or equivalently, orbit frameworks on a {\it fixed torus}, which we denote $\Tor^d$. By this we mean the quotient space $\Tor^d = \mathbb R^d/ L \mathbb Z^d$, where $L$ is a $d \times d$ matrix we call the {\it lattice matrix}. The lattice matrix can be viewed as a set of translations under which a periodic framework is invariant. In fact, it has been demonstrated that infinitesimal rigidity is invariant under affine transformations \cite{BS1,myThesis}, and as a result it is sufficient to consider frameworks on the unit lattice ($L = I_{d \times d}$) as representatives of all periodic orbit frameworks. We henceforth use $\Tor^d = \mathbb R^d/\mathbb Z^d = [0,1)^d$ as the fixed torus. 

When we assign specific geometric positions on the fixed torus $\Tor^d$ to the vertices of $\pog$, we call the resulting object a {\it periodic orbit framework (on $\Tor^d$)}, which we denote $\pofw$, with $p: V(G) \rightarrow \Tor^d$, and $p(v_i) \neq p(v_j)$ if $e = \{v_i, v_j; m_e\} \in E\pog$. This in turn corresponds to the {\it derived periodic framework}, $\dpfw$, where the positions of the vertices are given by 
\[p^m(v_i, z) = p(v_i) + z, \ \textrm{where }\ v_i \in V(G), z \in \mathbb Z^d.\]

In \cite{ThesisPaper1} we showed that every infinite periodic framework in the sense of Borcea and Streinu \cite{BS1} can be represented as the derived periodic framework corresponding to a periodic orbit framework on the torus. The approach differs in terms of motions, which we shall now describe.

 An {\it infinitesimal motion} of a periodic orbit framework $\pofw$ on $\Tor^d$ is an assignment of velocities to each of the vertices, $\bu: V(G) \rightarrow \mathbb{R}^d$,  such that 
 \begin{equation*}
(\bu_i - \bu_j)\cdot(\p_i - \p_j-\bbm_e) = 0 
\label{eqn:fixTorMot}
\end{equation*}
for each edge $e = \{v_i, v_j;\bbm_e\} \in E\pog$.
 Such an infinitesimal motion instantaneously preserves the lengths of all of the bars of the framework. An infinitesimal motion is called  {\it trivial} if it in fact satisfies:
\begin{equation*}
(\bu_i - \bu_j)\cdot(\p_i - \p_j-\bbm_e) = 0
\label{eqn:fixTorTrivMot}
\end{equation*}
for {\it all} triples  $\{v_i, v_j;\bbm_e\}$, $m_e \in \mathbb Z^d$.  For any periodic orbit framework $\pofw$ on $\Tor^d$, there will always be a $d$-dimensional space of  trivial infinitesimal motions of the whole framework, namely the space of infinitesimal translations.  Rotations are no longer trivial motions, since we have fixed a representation of the lattice (and therefore of the torus $\Tor^d$).

Suppose $u$ is an infinitesimal motion of $\pofw$ on $\Tor^d$. We can translate this to an infinitesimal motion of $\dpfw$ in $\mathbb R^d$ in the following way. Let $u^m: V(G) \times \mathbb Z^d \rightarrow \mathbb R^d$ be given by
\[u^m(v_i, z) = u(v_i).\]
In other words, all vertices in the orbit corresponding to the vertex $v_i \in V(G)$ have the same infinitesimal motion. This is in contrast to motions of $\pofw$ on a variable torus. In that setting, the motion assignment on $(i, z)$ will depend on $z \in \mathbb Z^d$. In particular, the velocities of the vertices of $\dpfw$ are allowed to be infinitely large. 

The vector space of all infinitesimal motions of $\pofw$ on $\Tor^d$ is given by the kernel of the {\it periodic rigidity matrix}, $\R(\pofw)$. This is the $|E| \times d|V|$ matrix with one row for each edge $e = \{v_i, v_j; m_e\} \in E\pog$ as follows: 
\[\left(\begin{array}{ccccc}0 \dots 0 & \p_i - (\p_j+m_e) & 0 \dots 0 & (\p_j+m_e)-\p_i & 0\dots 0\end{array}\right),\]
where the non-zero vector entries occur in the columns corresponding to vertices $v_i$ and $v_j$ respectively. 

There is always a $d$-dimensional space of trivial solutions to this matrix, namely the infinitesimal translations.  We have the following version of Theorem \ref{thm:finiteMatrix} for periodic orbit frameworks: 
\begin{thm}[\cite{ThesisPaper1}]
A periodic orbit framework $\pofw$ is infinitesimally rigid on the fixed torus $\Tor^d$ if and only if the rigidity matrix $\R\pofw$ has rank $d|V(G)| - d$.
\label{thm:periodicMatrix}
\end{thm}
If $\pofw$ is infinitesimally rigid on $\Tor^d$, and satisfies $|E(G)| = d|V(G)| - d$, we say that $\pofw$ is {\it minimally rigid on $\Tor^d$}.

If a set of rows of the rigidity matrix are linearly independent, we say that the corresponding edges are {\it independent}. Note then that minimal rigidity corresponds to the situation where the framework is infinitesimally rigid, {\it and} all edges of the framework are independent. For finite frameworks this combination is known as isostatic, however, we avoid the term isostatic in the periodic context, for the reasons discussed in \cite{DeterminancyRepetitive}. 

If a framework $\pofw$ contains loop edges, then those edges are automatically dependent. For this reason we usually assume that periodic bar-joint orbit graphs $\pog$ do not possess loops, since they are always {\it redundant}, meaning dependent in $\R(\pofw)$. 

As for finite frameworks, if $G$ is disconnected then $\pog$ is automatically infinitesimally flexible on $\Tor^d$. Note, however, that $G$ may be connected, but the derived periodic framework $G^m$ may be disconnected, as in the case of the two-vertex, two-edge graph where the edges are labelled by $(0,0)$ and $(1,0)$ respectively. Despite the disconnection of $G^m$, the periodic orbit graph $\pog$ is nevertheless infinitesimally rigid on the two-dimensional fixed torus $\Tor^2$, which provides an indication of the structure forced by working on the fixed torus. 

Similar to the situation for finite bar-joint frameworks, we may define a notion of a {\it generic} position of the vertices of $\pofw$ on $\Tor^d$. That is, we say that $p$ is generic if all minors of $\R(\pofw)$ that have determinants which are not identically zero have non-zero determinant. We say that the periodic orbit graph $\pog$ is {\it generically rigid}  if it is infinitesimally rigid at all generic positions $p$ on $\Tor^d$. 

We remark that to record a rigidity matrix for frameworks with a variable lattice, one must add extra columns. See \cite{BS1}, \cite{theran} or \cite{myThesis} for details. 

% FINITE BODY BAR FRAMEWORKS
\subsection{Finite body-bar frameworks}

Roughly speaking, a (finite) body-bar framework is a collection of bodies linked together with rigid bars. In defining body-bar frameworks more precisely, we follow the approach of Connelly, Jord\'{a}n and Whiteley \cite{genericGlobal}. The bodies of a body-bar framework consist of collections of vertices, with each body joined to some of the other bodies by disjoint bars. Each body is assumed to be an isostatic finite bar-joint framework, although the details of the specific isostatic frameworks are unimportant, provided that the framework spans an affine space of dimension $d-1$ or greater. A {\it generic} body-bar framework is one where all of the vertices of all of the bodies are generic, in the sense of Section \ref{sec:finiteBarJoint}. A multigraph $H$ records the connections between the bodies ($H$ is permitted multiple edges, but no loops), and each vertex of $H$ represents a body. 

When we wish to explicitly refer to the vertices of a particular body, we use the graph $G_H$, which is the multigraph $H$ in which each vertex has been replaced with an isostatic graph, creating an isostatic bar-joint framework. $G_H$ is indeed a simple graph with no multiple edges, because any two edges connecting a pair of bodies must have distinct vertices. Since every finite body in $\mathbb R^d$ has ${d+1 \choose 2}$ degrees of freedom, and there is a ${d+1 \choose 2}$-dimensional space of trivial motions of any body-bar framework, a body-bar framework is called {\it minimally rigid} if it satisfies $|E(H)| = {d+1 \choose 2}|V(H)| - {d+1 \choose 2}$. The following result of Tay characterizes the generic rigidity of body-bar orbit frameworks in $d$-dimensions:
\begin{thm}[\cite{Tay2}]
A multigraph $H$ is generically rigid as a body-bar framework if and only if the edges of $H$ admit a decomposition into ${d+1 \choose 2}$ edge-disjoint spanning trees. 
\end{thm}

%BAR BODY ORBIT FRAMEWORKS
\subsection{Body-bar periodic orbit frameworks}

We now define our main object of interest, body-bar periodic orbit graphs, which provide a recipe for periodic body-bar frameworks. Let $H$ be a multigraph, which is now permitted to have both multiple edges and loops. Let $m: E(H)^+ \rightarrow \mathbb Z^d$ be a map on the directed edges of $H$. The pair $\bbog$ is called a {\it body-bar periodic orbit graph}, where the labeled edges provide a description of how the edges of $H$ ``wrap" around the $d$-dimensional fixed torus $\Tor^d$ (see Figure \ref{fig:bodyBarEx}(a)). The vertices of $\bbog$ are denoted $B_1, B_2, \dots$ and the edges are given by $e=\{B_{\alpha}, B_{\beta}; m_e\}$.

\begin{figure}[htbp]
\begin{center}
\subfigure[$\bbog$]{\includegraphics[width=1.5in]{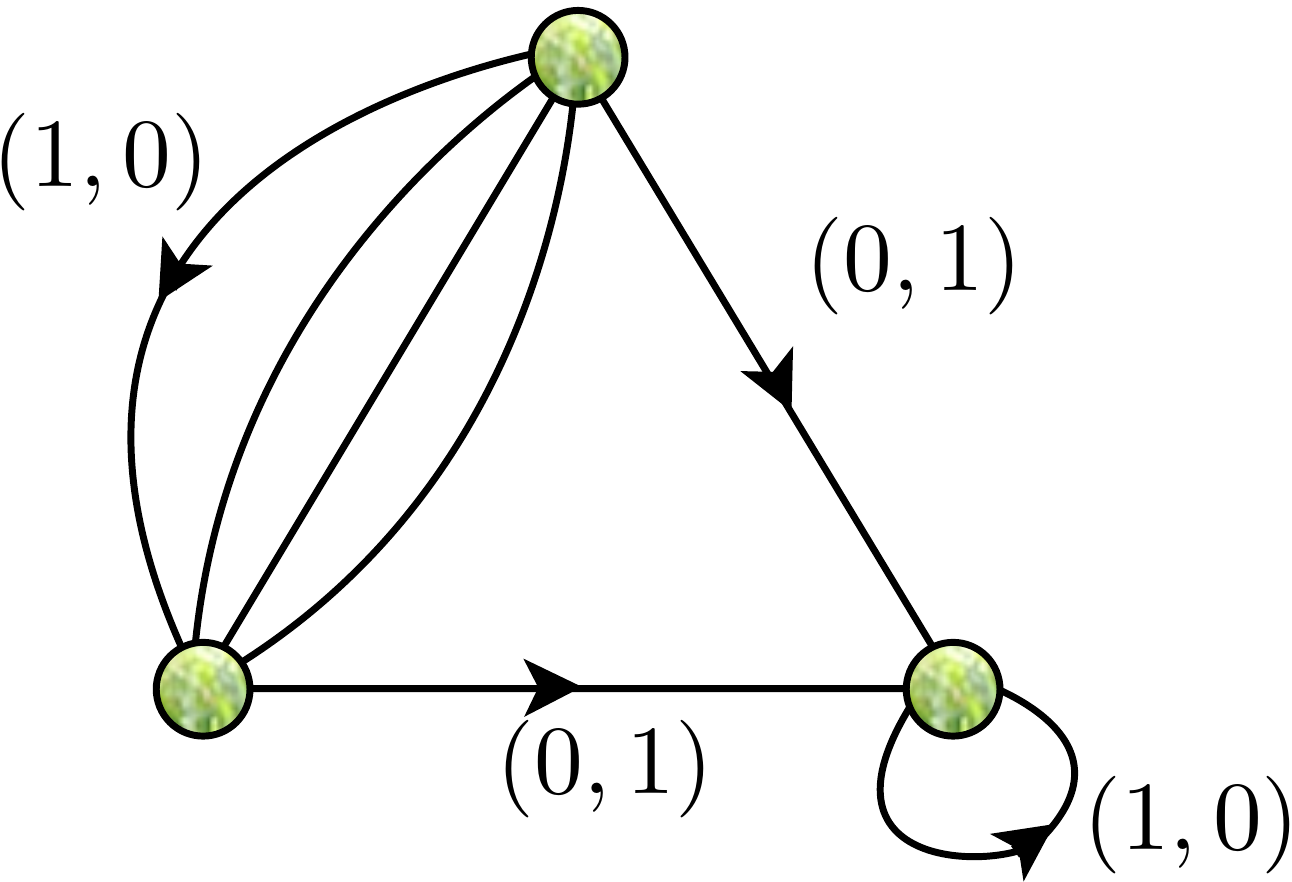}}\hspace{.25in}
\subfigure[$(H^m, q^m)$ in $\mathbb R^2$]{\includegraphics[width=2in]{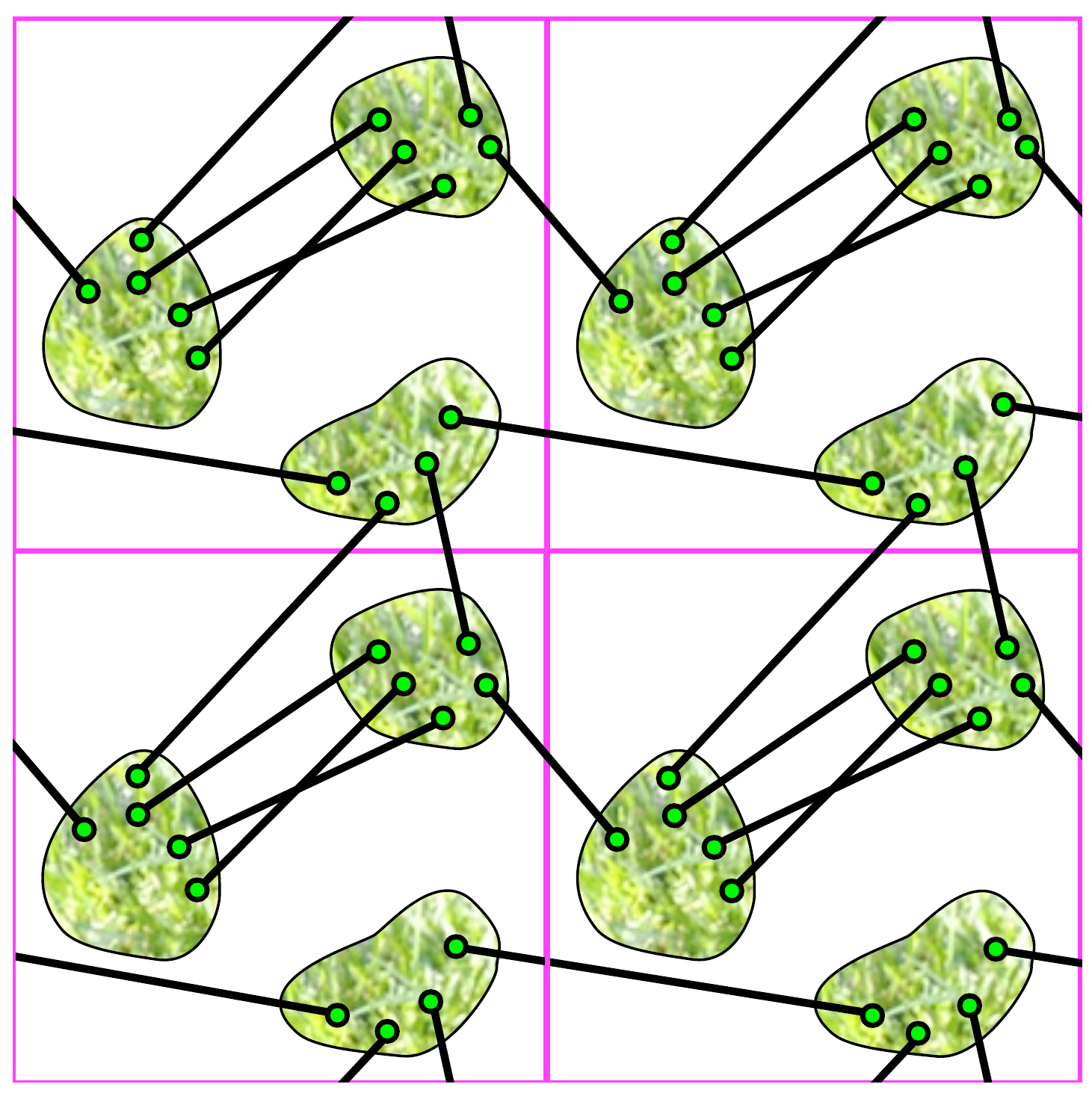}}
\caption{A two-dimensional body-bar periodic orbit graph, $m: E(H)^+ \rightarrow \mathbb Z^2$ (a), and a fragment of the corresponding derived periodic body-bar framework (b).}
\label{fig:bodyBarEx}
\end{center}
\end{figure}

Just as for bar-joint periodic orbit graphs, for a periodic body-bar orbit graph $\bbog$ we may define a derived periodic body-bar graph $H^m$ which is an infinite periodic structure (Figure \ref{fig:bodyBarEx}(b)). Suppose $\bbog$ is a periodic body-bar orbit graph. We define the {\it derived periodic body-bar graph} $H^m$ to be the infinite graph with vertex set $V(H) \times \mathbb Z^d$ and edge set $E(H) \times \mathbb Z^d$. If $e= \{B_{\alpha}, B_{\beta}; m_e\}$ is an edge of $E\bbog$, then the countably infinite set of edges
\[\{(B_{\alpha}, z), (B_{\beta}, z+m_e)\}, z \in \mathbb Z^d\]
forms an {\it orbit} of edges of $H^m$, connecting the orbit of the body $B_{\alpha}$ (denoted $(B_\alpha, z), z \in \mathbb Z^d$) with the appropriate copy (depending on $m_e$) of the orbit of the body $B_{\beta}$.

When we wish to explicitly refer to the underlying bar-joint framework, we denote by $\langle G_H, m_H \rangle$ the (bar-joint) periodic orbit framework obtained by replacing every vertex of $H$ with a finite (i.e. not periodic) isostatic bar-joint framework. The map $m_H$ will simply be $m$ on the edges of $H$, and will map all other edges to the zero vector (these are the edges of the isostatic frameworks making up each body (vertex) of $H$).  

{\it Body-bar periodic orbit frameworks} $\bbofw$ are body-bar orbit graphs, together with a map $q$, which maps the vertices of $\bbog$ to {\it bodies} on $\Tor^d$, and the edges of $\bbog$ to bars on $\Tor^d$. We assume that each body spans an affine subspace of $\mathbb R^d$ of dimension at least $d-1$, and furthermore that, up to translation, the body lies completely within the interval $[0,1)^d (= \Tor^d)$. This can also be expressed as a map $p_H$ directly on the vertices of the induced bar-joint framework $\pogind$. 

An {\it infinitesimal motion} of the body-bar periodic orbit framework $\bbofw$ can be defined as an infinitesimal motion of the underlying bar-joint framework $\pofwind$. That is, it is a map $u: V(G_H) \rightarrow \mathbb R^d$ such that the lengths of all edges $E\pogind$ are infinitesimally preserved:
		\[(u_{\alpha, i}- u_{\beta, j})\cdot(p_{\alpha, i}- p_{\beta, j} -m_{e})=0,\]
		for all edges $e = \{v_{\alpha, i}, v_{\beta, j}; m_{e}\}$ in $E\pogind$ (which corresponds in turn to the edge $\{B_{\alpha}, B_{\beta}; m_{e}\}$ in $E\bbog$).

An infinitesimal motion of $\bbofw$ on $\Tor^d$ is {\it trivial} if it is an infinitesimal translation. As for bar-joint periodic orbit frameworks, rotations are again not considered trivial motions in this conception, as described in Section \ref{sec:barJointOrbit}. $\bbofw$ is {\it infinitesimally rigid} on $\Tor^d$ if $(\pogind, p_H)$ is infinitesimally rigid on $\Tor^d$, that is, if its only infinitesimal motions are trivial (translations). 

Every body in $\mathbb R^d$ has ${d+1 \choose 2}$ degrees of freedom. Similar to the situation for bar-joint periodic orbit frameworks on $\Tor^d$, there is always a $d$-dimensional space of trivial motions (corresponding to the unit translations in $d$ directions). It follows that for minimal rigidity of a body-bar periodic orbit framework on $\Tor^d$, we require exactly $|E(H)| = {d+1 \choose 2}|V(H)| - d$ edges. 

\begin{prop}
Let $\bbofw$ be a $d$-periodic body-bar framework, with $|E(H)| = {d+1 \choose 2}|V(H)| - d$. Then the induced bar-joint framework $\pofwind$ will have $|E(G_H)| = d|V(G_H)| - d$. 
\label{prop:bbToBarJointComb}
\end{prop}
The proof is elementary.

%\begin{proof}
%For each body $B \in V$ we insert an isostatic framework $(V_B, E_B)$ satisfying $|E_B| = d|V_B| - {d+1 \choose 2}$. For ease of notation we write $G = G_H$. Then 
%\[|V(G)| = \sum_{i=1}^{|V(H)|} |V_i|, \textrm{\ and\ }\]
%\begin{eqnarray*}
%|E(G)| & = & |E(H)| + \sum_{i=1}^{|V(H)|}|E_i|\\
%& = & {d+1 \choose 2}|V(H)| - d + \sum_{i=1}^{|V(H)|}|E_i|\\
%& = & {d+1 \choose 2}|V(H)| - d + \sum_{i=1}^{|V(H)|}(d|V_i| - {d+1 \choose 2})\\
%& = & {d+1 \choose 2}|V(H)| - d  + d\sum_{i=1}^{|V(H)|}|V_i|  - {d+1 \choose 2}|V(H)|\\
%& = & d\sum_{i=1}^{|V(H)|}|V_i|  - d\\
%& = & d|V(G)|  - d.\\
%\end{eqnarray*}
%\end{proof}

\begin{rem}
Instead of replacing each vertex of $H$ by an isostatic finite framework, we could alternatively replace each vertex of $H$ with a minimally rigid periodic orbit framework. In this case, the new gain assignment $\bbm_H$ would inherit the gains on the periodic orbit frameworks corresponding to each body. We do not pursue this variation in the present work.
\end{rem}

We say that the body-bar orbit framework $\bbofw$ is {\it generic} on $\Tor^d$ if the induced bar-joint framework $(\pogind, p_H)$ is generic. That is, the attachment vertices (of the endpoints of the edges of $H$) and all other vertices added as part of the isostatic frameworks replacing the vertices of $H$ are generic in the sense of bar-joint frameworks. 

We say that the edges of a body-bar orbit graph $\bbog$ are (generically) {\it independent} if the corresponding edges of $\pogind$ are independent as rows in the bar-joint periodic orbit matrix $\R(\pogind, p_H)$ for generic positions $p_H$. Note that the edges of any of the isostatic frameworks which make up the bodies will automatically be independent, since they correspond to generic isostatic frameworks. 

In contrast to the situation for bar-joint periodic orbit frameworks, loops may be permitted to be independent in body-bar periodic orbit frameworks. The endpoints of the loop must correspond to distinct vertices in the underlying bar-joint framework (see Figure \ref{fig:bodyBarEx}(b)). Each body may have up to ${d+1 \choose 2} - d$ independent loops on $\Tor^d$. Note also that for bar-joint periodic orbit frameworks on a variable torus, loops are also possibly independent, but with more restrictions \cite{theran,myThesis}.

\begin{rem}
It may be desirable to define a rigidity matrix for periodic body-bar frameworks directly, without reverting to a bar-joint framework. One possible version uses the language of extensors and exterior algebra, as in White and Whiteley \cite{PureCondition}, or in the coordinatized presentation in \cite{JJBarBody}. A rigidity matrix for body-bar frameworks with a flexible lattice was recently described in \cite{PeriodicBarBody}, and a matrix for the fixed lattice could be deduced from this presentation if desired. 
\end{rem}

%INDUCTIVE CONSTRUCTIONS ON BAR JOINT
\subsection{Inductive constructions on 3-dimensional bar-joint periodic orbit frameworks}
\label{sec:barJointInductions}

In \cite{GeneratingIsostaticFrameworks}, Tay and Whiteley outline several methods for generating generically rigid bar-joint frameworks (finite frameworks). Some of the original ideas and definitions date back to Henneberg \cite{Henneberg}, who proved that all generically rigid frameworks in $\mathbb R^2$ can be generated by a sequence of vertex additions and edge splits. Building on the work of Tay and Whiteley, we now define periodic versions of these two moves, for $3$-dimensional bar-joint periodic orbit frameworks. \\

%vertex addition
\noindent {\bf Vertex addition.}

Let $\pog$ be a (bar-joint) periodic orbit graph. A {\it periodic vertex addition} (Figure \ref{fig:vertexAdd}) is the addition of a single new vertex $v_0$ to $V\pog$, and the edges 
\[Y = \{v_0, v_{i_1}; m_{01}\}, \{v_0, v_{i_2}; m_{02}\}, \{v_0, v_{i_3}; m_{03}\}\] 
to $E\pog$, subject to the following restrictions:
\begin{enumerate}[(i)]
	\item $m_{0j} \neq m_{0k}$ whenever $v_{i_j} = v_{i_k}$ (i.e. if $v_{i_j} = v_{i_k}$, then $|\gs(Y)| \geq 1$),
	\item if $v_{i_1}=v_{i_2}=v_{i_3}$, then $m_{i_j} \neq m_{i_k}$, and $|\gs(Y)| \geq 2$. 
\end{enumerate}
\begin{figure}[htbp]
\begin{center}
\includegraphics[width=2.5in]{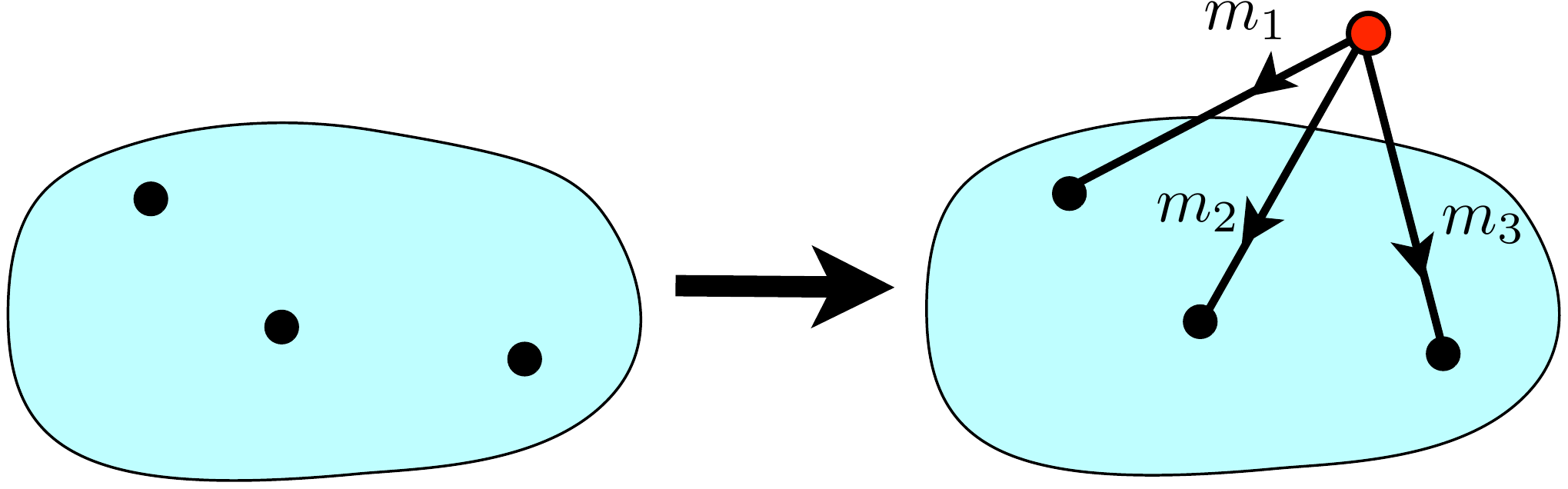}
\includegraphics[width=1in]{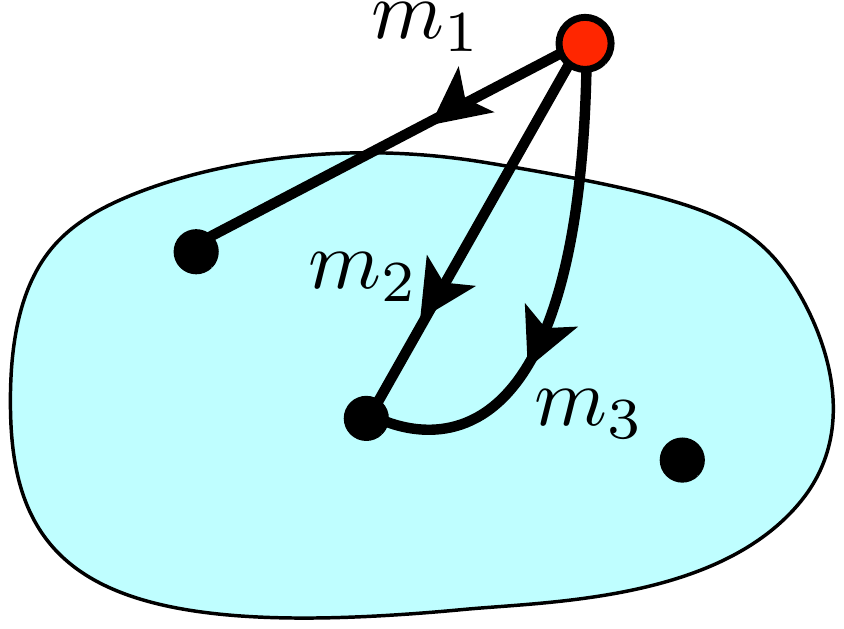}
\includegraphics[width=1in]{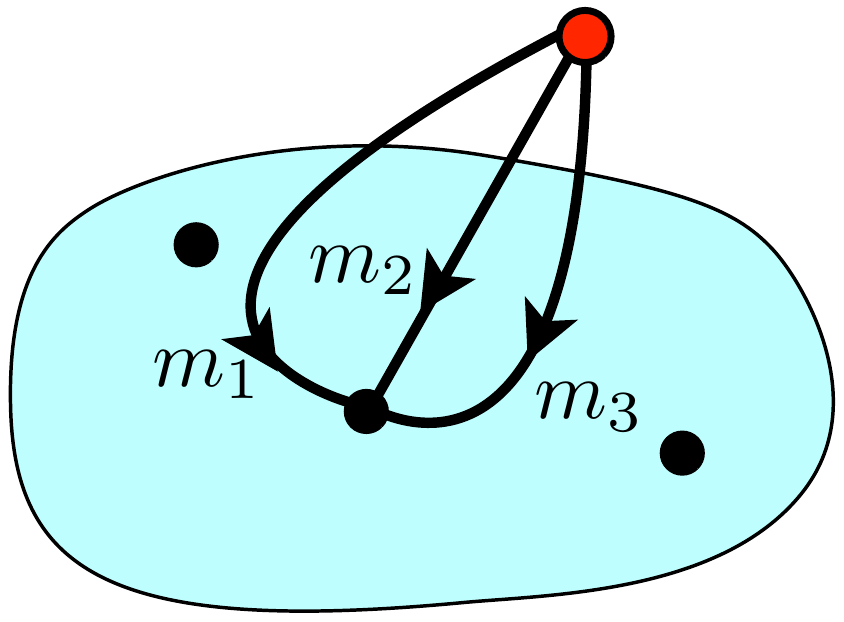}
\caption{Vertex addition on a bar-joint framework in $\mathbb R^3$. The three new edges may be connected to one, two or three distinct vertices in $G$.}
\label{fig:vertexAdd}
\end{center}
\end{figure}

Note that condition (i) simply ensures that no two edges connecting a single pair of vertices have the same gains, which would correspond to a pair of identical edges in $G^m$. The second restriction ensures that if all three edges connect vertex $v_0$ with a single other vertex in $\pog$, then the vertex addition does not produce a framework with degenerate geometry. For example, if the three gains $m_1, m_2, m_3$ are $(1,0,0), (2,0,0)$ and $(3,0,0)$, then the vertex addition will always add a vertex which is incident to a line, and therefore has a non-trivial motion. However, $(1,0,0), (2,0,0)$ and $(0,1,0)$ would be an example of a valid gain assignment. 

\begin{prop}
Let $\pog$ be a (bar-joint) periodic orbit graph, with $m:E^+ \rightarrow \mathbb Z^3$, and let $\pogp$ be the orbit graph created by performing a vertex addition on $\pog$, adding the vertex $v_0$. If $\pog$ is generically rigid on $\Tor^3$ then so is $\pogp$. 
\end{prop}

The proof of this fact is a straightforward extension of the proof of the two-dimensional version presented in \cite{myThesis}, and we omit it.\\

%edge splitting
\noindent {\bf Edge splitting.}\\
Let $\pog$ be a periodic orbit graph, and let $e = \{v_{1}, v_{2}; m_e\}, m_e \in \mathbb Z^3$ be an edge in $E\pog$. A {\it periodic edge split} $\pogp$ of $\pog$ is a graph with vertex set $V\pog \cup \{v_0\}$ and edge set consisting of all of the edges of $E\pog$ except $e$, together with the four additional edges
\[\{v_1, v_0; (0,0,0)\}, \{v_0, v_2; m_e\}, \{v_0, v_{i_1}; m_{01}\}, \{v_0, v_{i_2}; m_{02}\},\]
where the gains are subject to the same restrictions as for vertex additions (Figure \ref{fig:edgeSplit}). Precisely:
\begin{enumerate}[(i)]
	\item if any two of the added edges have the same endpoints, their gains are distinct;
	\item if any three of the added edges have the same endpoints, then each pair of edges satisfies (i), and the gain space of the three edge, two vertex graph has dimension at least 2. 
\end{enumerate}
\begin{figure}[htbp]
\begin{center}
\includegraphics[width=3.5in]{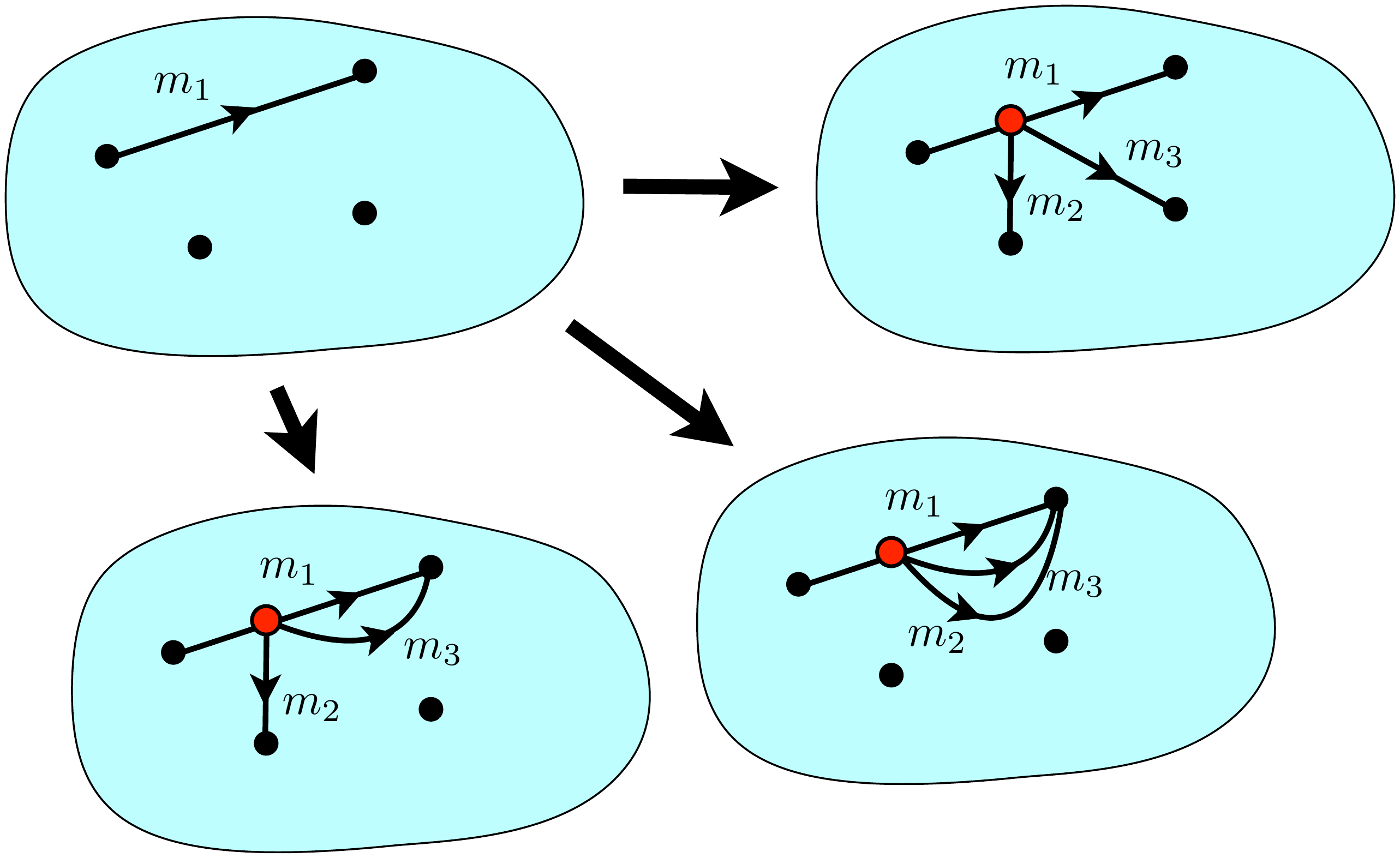}
\caption{Edge split on a single edge of a bar-joint framework in $\mathbb R^3$. The added edges may connect with distinct vertices, or the end-points of the split edge, provided that we do not introduce redundant edges.  }
\label{fig:edgeSplit}
\end{center}
\end{figure}

\begin{prop}
Let $\pog$ be a periodic orbit graph, with $m:E^+ \rightarrow \mathbb Z^3$, and let $\pogp$ be the orbit graph created by performing an edge split on $\pog$, adding the vertex $v_0$. If $\pog$ is generically rigid on $\Tor^3$ then so is $\pogp$. 
\end{prop}

\begin{proof}

Let $p$ be a generic position of $\pog$, and let $(G^m, p^m)$ be the derived periodic framework in $\mathbb R^3$. Suppose we wish to perform the edge split on the edge $e = \{v_{1}, v_{2}; m_e\}$, which would add the edges listed above. 

We begin by performing a vertex addition on $\pog$, adding vertex $v_0$ and the three edges  
\[Y= \{v_1, v_0; (0,0,0)\}, \{v_0, v_{i_1}; m_{01}\}, \{v_0, v_{i_2}; m_{02}\}.\]
We select a position $p_0$ for the vertex $v_0$ on $\Tor^3$ so that $p_0$ lies on the line of the edge $e$ on $\Tor^3$. Equivalently, the orbit of vertices $(v_0, z), z \in \mathbb Z^3$ in $V(G^m, p^m)$ lie along the lines of the orbit of edges $(e, z)$ in $E(G^m, p^m)$. 

We now use Proposition \ref{prop:triangleReplacement} to replace the pair of edges 
\[e = \{v_{1}, v_{2}; m_e\}, \{v_1, v_0; (0,0,0)\}\]
with the pair
\[\{v_{0}, v_{2}; m_e\}, \{v_1, v_0; (0,0,0)\},\]
which maintains generic rigidity and completes the proof. 
\end{proof}

The following proposition is a restricted version of the Isostatic Substitution Principle, see \cite{GeneratingIsostaticFrameworks} for example. It is sometimes referred to as the ``Triangle Exchange," from the operation on a collinear triangle. 

\begin{prop}
Let $\pofw$ be a generically rigid periodic orbit framework on $\Tor^3$, with edge $e = \{v_{1}, v_{2}; m_e\}$. Let $p_0$ be a point on the line connecting the positions $p_1$ and $p_2+m_e$. Then replacing the pair of edges 
\[e = \{v_{1}, v_{2}; m_e\}, \{v_1, v_0; (0,0,0)\}\]
with the pair
\[\{v_{0}, v_{2}; m_e\}, \{v_1, v_0; (0,0,0)\},\]
will result in another generically rigid periodic orbit framework on $\Tor^3$
\label{prop:triangleReplacement}
\end{prop}

\begin{proof}
Since we have chosen the point $p_0$ so that the triple of edges 
\[e = \{v_{1}, v_{2}; m_e\}, \{v_1, v_0; (0,0,0)\}, \{v_{0}, v_{2}; m_e\}\]
is collinear, these edges therefore form a dependent set of rows in the rigidity matrix. Hence we may replace any pair for any other pair. 
\end{proof}

\begin{rem}
The preceeding definitions of vertex addition and edge split are natural in the following sense: they correspond to ``classical" vertex additions and edge splits on the (infinite) derived periodic framework $(G^m, p^m)$, in which each move adds/modifies whole orbits of vertices and edges. In addition, Proposition \ref{prop:triangleReplacement} corresponds to replacing whole orbits of edges with other (equivalent) orbits in $(G^m, p^m)$. 
\end{rem}

%MAIN RESULT
\section{Statement of main result}
\label{sec:mainResult}

Before we state our main result, we require one final definition, namely that of the gain space of a gain graph. 

The {\it cycle space} $\C = \C(G)$ of $G$ is the subspace of the edge space of a graph $\mathcal E(G)$ (see \cite{Diestel}) spanned by the (edge sets of the) cycles of $G$. Suppose $\pog$ is a gain graph where $\C(G)$ is the cycle space of the (undirected) graph $G$. The {\it net gain} on a cycle of $G$ is determined by summing the gains on each edge of a participating edge, taking the positive or negative gain depending on the direction of traversal. The {\it gain space} $\gs(G)$ is the vector space (over $\mathbb Z$) spanned by the net gains on the cycles of $\C(G)$. Depending on the context, we will either write $\gs(G)$ or $\gs(Y)$, where $Y \subseteq E(G)$ is a subset of the edge set of $G$. 

In contrast to cycles in directed graphs, we permit re-direction of the edges of a gain graph provided that they are accompanied by a relabelling of the gains on the edges as well. In this way, we should think of cycles in the gain graph as corresponding one-to-one with cycles in the base graph.

We may now state our main result.
\begin{thm}
$\bbog$ is a generically minimally rigid body-bar periodic orbit graph on $\Tor^3$ if and only if 
\begin{enumerate}
	\item $|E(H)| = 6|V(H)| - 3$
	\item for all non-empty subsets $Y \subset E(H)$ of edges
		\begin{equation}
			|Y| \leq 6|V(Y)| - 6 + \sum_{i=1}^{|\gs(Y)|}(3-i). \tag{$\ast$}
		\label{eq:gainSparse}
		\end{equation}	
\end{enumerate}
\label{thm:main}
\end{thm}

We remark that we have chosen to state this result in terms of edge-induced subgraphs to distinguish it from the usual approach in rigidity theory of using vertex-induced subgraphs. 

Observe that the first two terms of the sparsity expression (\ref{eq:gainSparse}) correspond to necessary conditions for the independence of a finite body-bar framework. Recall that for such a framework to be independent we must have that $|Y| \leq 6|V(Y)| - 6$ for all $Y \subseteq E(G)$. In this way, (\ref{eq:gainSparse}) indicates that we may add additional edges (given by the term $\sum(3-i)$), depending on the dimension of the gain space, $|\gs(Y)|$. Furthermore, it is clear that any graph satisfying conditions 1 and 2 will have the property that for all vertex-induced subgraphs, for $X \subseteq V(H)$, will satisfy $|E(X)| \leq 6|X| - 3$. That is, $H$ is a $[6,3]$-graph (see Section \ref{sec:FS}). 

To further motivate this result, we consider an example. 
\begin{ex}
For finite body-bar frameworks, there may be at most six edges between any two bodies (see Figure \ref{fig:twoBodyExample}(a)). For body-bar periodic orbit frameworks, there may be at most {\it nine} edges between any two bodies (see Figure \ref{fig:twoBodyExample}(b)). By (\ref{eq:gainSparse}), labels on these edges must have the property that $|\langle m_1, m_2, m_3 \rangle | \geq 2$, and $|\langle m_i, m_j \rangle| \geq 1$. For example, $(1,0,0), (1,0,0), (0,1,0)$ is a valid labelling of $m_1, m_2$ and $m_3$ respectively. In this case, any set $Y$ of seven or eight edges will have $|\gs(Y)|\geq 1$, and the full set of nine edges has $|\gs(E)| = 2$. 
\begin{figure}[htbp]
\begin{center}
\subfigure[finite body-bar framework]{\includegraphics[width=1in]{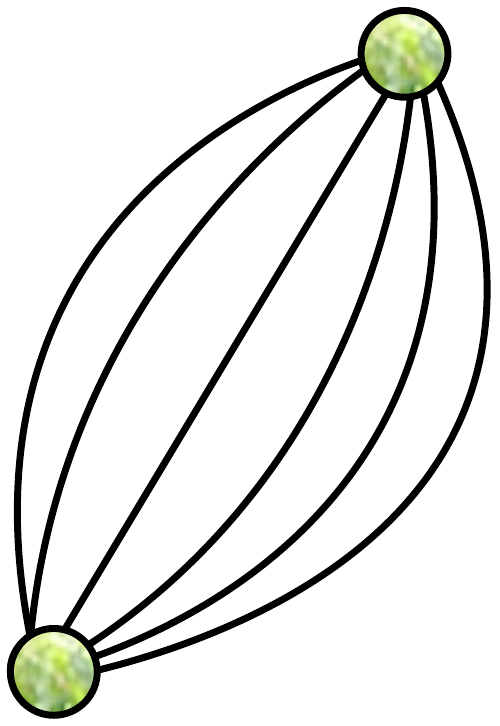}}\hspace{.5in}
\subfigure[body-bar periodic orbit framework]{\includegraphics[width=1.5in]{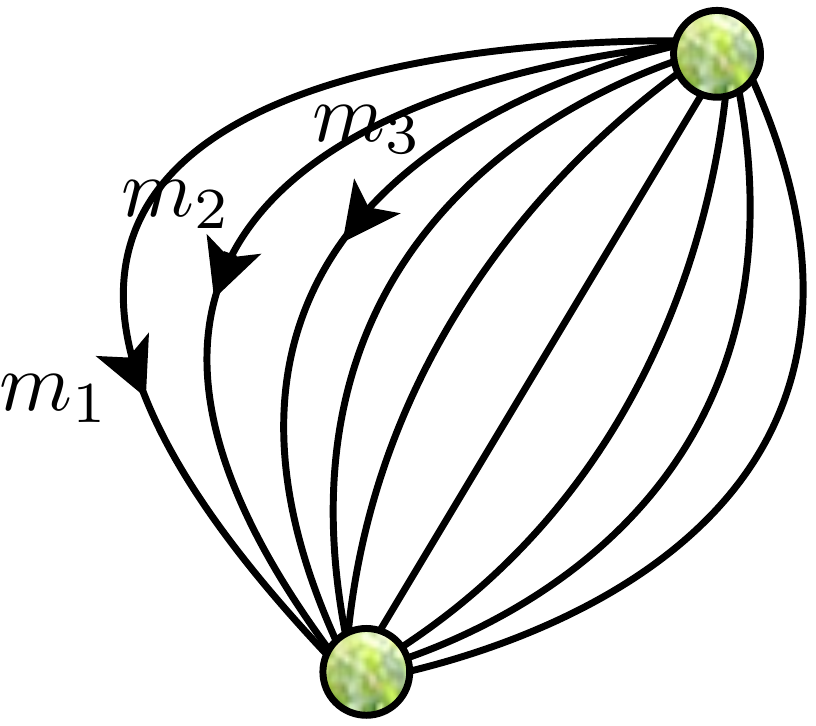}}
\caption{For finite body-bar frameworks in three dimensions, we may have at most six edges between any two bodies (a). For periodic orbit body-bar frameworks, we may have at most nine edges between any two bodies (b), provided the edges satisfy the sparsity condition (\ref{eq:gainSparse}).}
\label{fig:twoBodyExample}
\end{center}
\end{figure}
\qed
\end{ex}
\begin{rem} Counter-intuitively, we do not require that $|\gs(G)|=3$ for rigidity. In other words, it is possible for a disconnected derived periodic body-bar framework $(H^m, q^m)$ to be infinitesimally rigid in $\mathbb R^d$. This is due to the extra constraints imposed by the fixed lattice or torus. \end{rem}

%Necessary conditions
\subsection{Necessary conditions for generic rigidity on $\Tor^3$}
To see why the conditions of Theorem \ref{thm:main} are necessary, we use an earlier result from the rigidity of bar-joint periodic orbit frameworks on $\Tor^d$, as outlined by Theorem \ref{thm:dNecessary}. Recall that the minimal number of edges for a periodic orbit bar-joint framework to be infinitesimally rigid on $\Tor^d$ is $d|V| - d$. 
\begin{thm}[\cite{ThesisPaper1}]
Let $\pog$ be a minimally rigid bar-joint framework on $\Tor^d$. Then for all subsets of edges $Y \subseteq E$, 
\begin{equation}
|Y| \leq d|V(Y)| - {d+1 \choose 2} + \sum_{i=1}^{|\gs(Y)|} (d-i).
\label{eq:dNecessary}
\end{equation}
\label{thm:dNecessary}
\end{thm}

After converting our body-bar orbit graph $\bbog$ to a bar-joint orbit graph $\pogind$, we see that the necessary conditions for bar-joint and body-bar orbit graphs on $\Tor^d$ are equivalent. We remark that the first two terms in (\ref{eq:dNecessary}) are simply the necessary conditions for the independence of a finite periodic bar-joint framework in $d$-dimensions: $|Y| \leq d|V(Y)| - {d+1 \choose 2}$. When $|\gs(Y)| \geq d-1$, (\ref{eq:dNecessary}) becomes $|Y| \leq d|V(Y)| - d$.  See \cite{ThesisPaper1} for a detailed proof.

%SUFFICIENT CONDITIONS
\section{Sufficient conditions for generic rigidity on $\Tor^3$}
\label{sec:sufficient}

The proof of the sufficiency of Theorem \ref{thm:main} is more challenging. We use the following approach. Section \ref{sec:FS} describes a result of Fekete and Szeg\H{o} which provides an inductive characterization of the combinatorial class of graphs in which we are interested. Section \ref{sec:modEdgePinches} shows that modifying these inductive moves (edge pinches) to incorporate gains preserves the generic rigidity of the corresponding body-bar orbit frameworks on $\Tor^3$. Finally, section \ref{sec:combinatorialStep} provides the main combinatorial step, in proving that all gain graphs which satisfy the sparsity condition (\ref{eq:gainSparse}) can be generated through a sequence of gain-modified edge pinches from an appropriate starting graph, which is a direct modification of the techniques of Fekete and Szeg\H{o}. 

%Fekete Szego result
\subsection{Inductive constructions for $[6,3]$-graphs}
\label{sec:FS}

Let $H$ be a multigraph. For a subset of vertices $X \subset V(H)$, let $i_H(X)$ denote the number of induced edges on the vertex set $X$. We say that $H$ is $[k,\ell]$-sparse if $|i_H(X)| \leq k|X| - \ell$ holds for every $X \subset V(H)$. $H$ will be called a $[k, \ell]$-graph if, in addition to being $[k, \ell]$-sparse, $H$ also satisfies $|E(H)| = k|V(H)| - \ell$. In some terminology $[k, \ell]$-graphs are known as $[k, \ell]$-tight graphs, see \cite{LeeStreinu} for example. 

We summarize a result of Fekete and Szeg\H{o} \cite{fekete}, first by introducing inductive graph constructions called {\it edge pinches}. Let $0 \leq j \leq n \leq k$. An {\it edge pinch}, denoted $K(k,n,j)$, consists of the following: choose $j$ edges of $H$, and pinch them into a new vertex, $z$. Link $z$ with other vertices of $H$ with $k-n$ new edges, and place $n-j$ loops on the vertex $z$ (see Figure \ref{fig:edgePinch}). Following the edge pinch, the new graph has $k$ more edges than the original, and the new vertex has degree $k+n$. 

\begin{figure}[htbp]
\begin{center}
\includegraphics[width=4in]{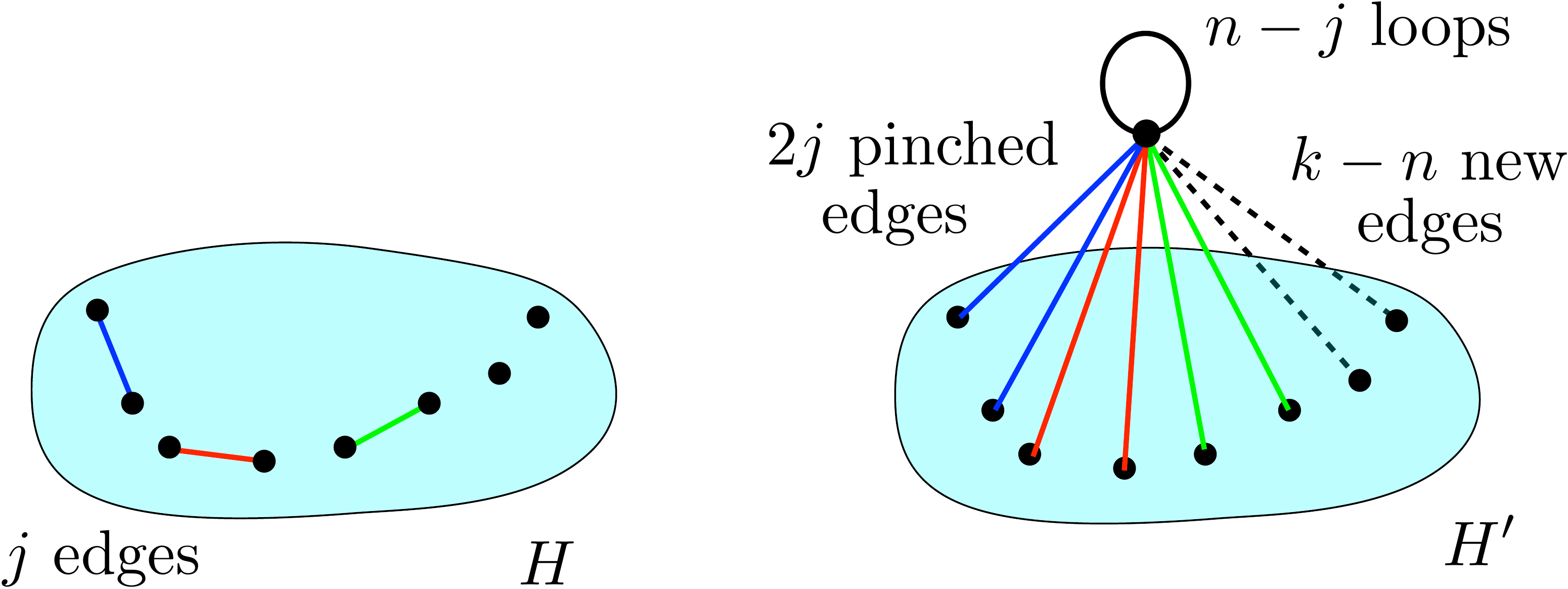}
\caption{An edge pinch $K(k, n, j)$ where $k=7$, $n=5$, $j=3$. The new vertex has degree $12$ (the loop contributes $2$ to the degree).}
\label{fig:edgePinch}
\end{center}
\end{figure}

The single vertex graph with $x$ loops will be denoted $P_x$. 

The main result of \cite{fekete} is the following: 
\begin{thm}[\cite{fekete}]
Let $H$ be a multigraph and let $1 \leq \ell \leq k$. Then $H$ is a $[k, \ell]$-graph if and only if $H$ can be created from $P_{k-\ell}$ with edge pinches $K(k, n, j)$, where $j \leq n \leq k-1$, and $n-j \leq k-\ell$. 
\label{thm:fekete}
\end{thm}

In Section \ref{sec:combinatorialStep}, we will state and prove a modified version of this result for $[k, \ell] = [6,3]$ that follows the proof of Fekete and Szeg\H{o}. We first describe gain-modified versions of edge pinches, and prove that they preserve the generic rigidity of body-bar periodic orbit frameworks on $\Tor^3$. 

We remark that Lee and Streinu independently characterized the same range of $[k,\ell]$-graphs, and also have a Henneberg-type result \cite{LeeStreinu}. We do not use their methods here. 

%Modified edge pinches preserve generic rigidity (Inductive step)
\subsection{Modified edge pinches preserve generic rigidity of body-bar periodic orbit frameworks in 3 dimensions}
\label{sec:modEdgePinches}

We now describe gain-modified versions of edge pinches in three-dimensions, which are edge pinches on body-bar orbit graphs $\bbog$ ($m:E^+ \rightarrow \mathbb Z^3$). A {\it gain-modified edge pinch} $K^*(6,n,j)$ on $\bbog$ is an edge pinch $K(6,n,j)$ on $H$ to form the graph $H^*$, together with an extension $m^*$ of the gain assignment $m$ to the new edges, such that 
\begin{enumerate}[(i)]
	\item any pinched edge $e = \{B_{\alpha}, B_{\beta}; m_e\}$ becomes the two edges $\{B_{\alpha}, B_0; m_{e_1}\}$ and $\{B_0, B_{\beta}; m_{e_2}\}$, where $m_{e_1} + m_{e_2} = m_e$;
	\item any subset of the new edges (including the pinched edges) satisfies (\ref{eq:gainSparse}). 
\end{enumerate}
Condition (ii) ensures that, for example, loops on the new vertex $z$ have gain assignments which satisfy the necessary conditions of Theorem \ref{thm:main}. 

\begin{thm}
Let $\bbog$ be a periodic body-bar orbit graph, and suppose $\bbogp$ is the orbit graph resulting from performing the edge pinch $K^*(6,n,j)$ on $\bbog$. If $\bbog$ is generically rigid on $\Tor^3$, then so is $\bbogp$. 
\label{thm:bbPinches}
\end{thm}

\begin{proof}
The proof follows the following two steps. It is related to the arguments appearing in the paper of Connelly, Jord\'{a}n and Whiteley \cite{genericGlobal}.

1. $j = 0$

2. $1 \leq j \leq 5$\\

{\bf Step 1. $j=0$ (body addition)} Note that when $j=0$, we are not splitting any edges. Instead, we are simply adding a new vertex to $H$ (a body), which has between 0 and 3 loops. For bar-joint frameworks, this type of move is usually called a {\it vertex addition} (see for example, \cite{GeneratingIsostaticFrameworks}), but for our purposes we will think of it as a {\it body addition}, since each vertex in $H$ represents a body in a periodic orbit framework $\bbofw$ on $\Tor^3$. 

Let $\bbog$ be the body-bar orbit graph, and suppose we are performing the edge pinch $K^*(6,n,0)$, where $0 \leq n \leq 3$.
We first convert $\bbog$ to the induced bar-joint framework $\pogind$. We now have four cases, for the four possible values of $n$, see Table \ref{tab:bodyAddition}. We perform rigidity-preserving moves on the bar-joint frameworks, as described in Section \ref{sec:barJointInductions}. The first column of Table \ref{tab:bodyAddition} is the ``target", meaning that it is this body-bar framework $\langle H', m' \rangle$ that we wish to create. The three other columns depict moves on the underlying bar-joint framework $\langle G_H, m_H \rangle$.

All four cases follow the same formula, (vertex addition, edge split, vertex addition) as illustrated in Table \ref{tab:bodyAddition}. In the case that $m=1$ or $2$, we can perform the pictured edge split, since we know that $m_6 \neq (0,0,0)$ (as the gain on the loop in our target). In the case that $m=3$, and we are performing the pinch $K^*(6,3,0)$, we note that among the three loops with gains $m_4, m_5, m_6$, we must have at least two distinct gains. That is, we may assume without loss of generality that $m_4 \neq m_5 \neq (0,0,0)$. This makes the edge split we perform a valid one. Let the bar-joint periodic orbit graph resulting from this sequence of moves be denoted $\langle G_H', m_H' \rangle$. 

The final step in all cases is to recognize the triangle with all edges having zero gains as an isostatic graph in $\mathbb R^3$ spanning a two-dimensional affine subspace, which permits us to treat this triangle as a body, and convert $\langle G_H', m_H' \rangle$ back to a body-bar framework $\langle H', m' \rangle$. That is, we add a new vertex to the the vertices of the graph $H$ to form $V(H')$, and any edges with non-zero gains that are adjacent only to this triangle become loops on the new body. Since we preserved the generic rigidity of the induced bar-joint framework at every stage, the resulting body-bar orbit graph $\langle H', m' \rangle$ is generically rigid on $\Tor^3$ too. \\

\begin{table}[htdp]
\caption{Body addition (edge pinches on zero edges). Unlabeled, undirected edges have gain $(0,0,0)$.}
\begin{center}
\begin{tabular}{|cc||c|c|c|}
\hline
\multicolumn{2}{|c||}{Body-bar orbit graph $\bbog$} & \multicolumn{3}{c|}{Bar-joint orbit graphs}\\
\hline
$K^*(6,0,0)$ & \includegraphics[width=1in]{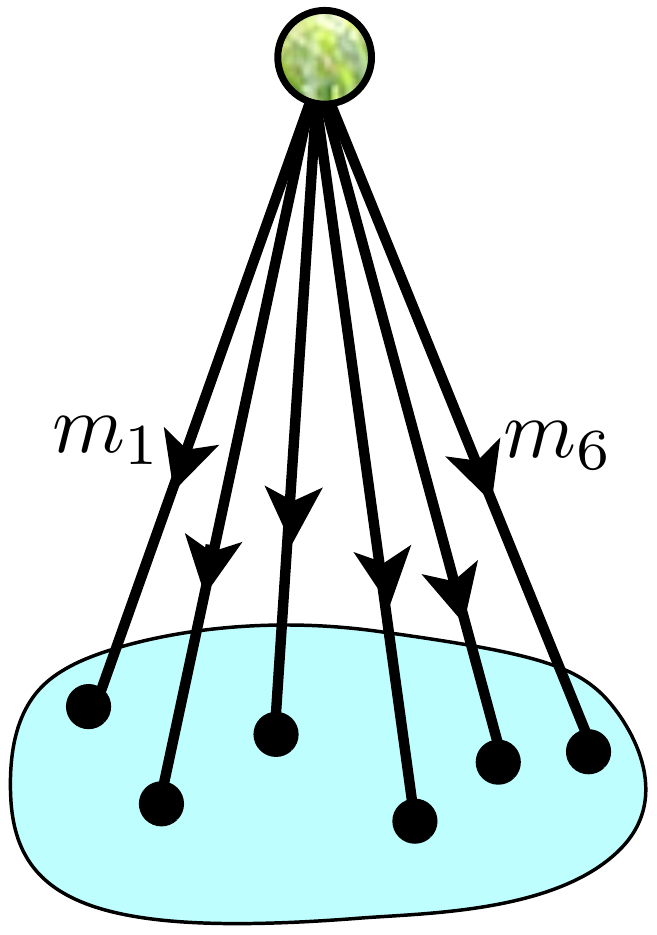} & \includegraphics[width=1in]{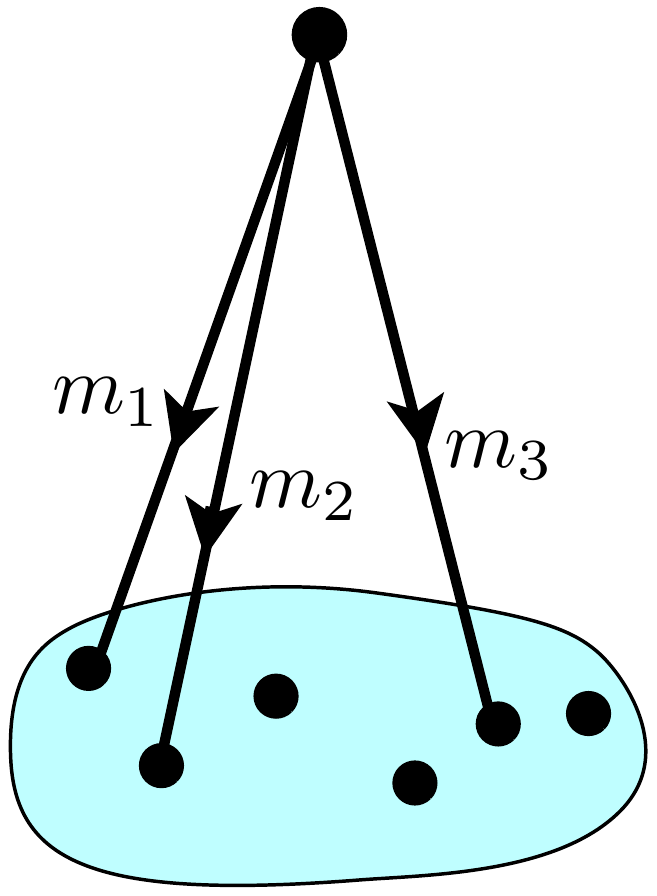} & \includegraphics[width=1in]{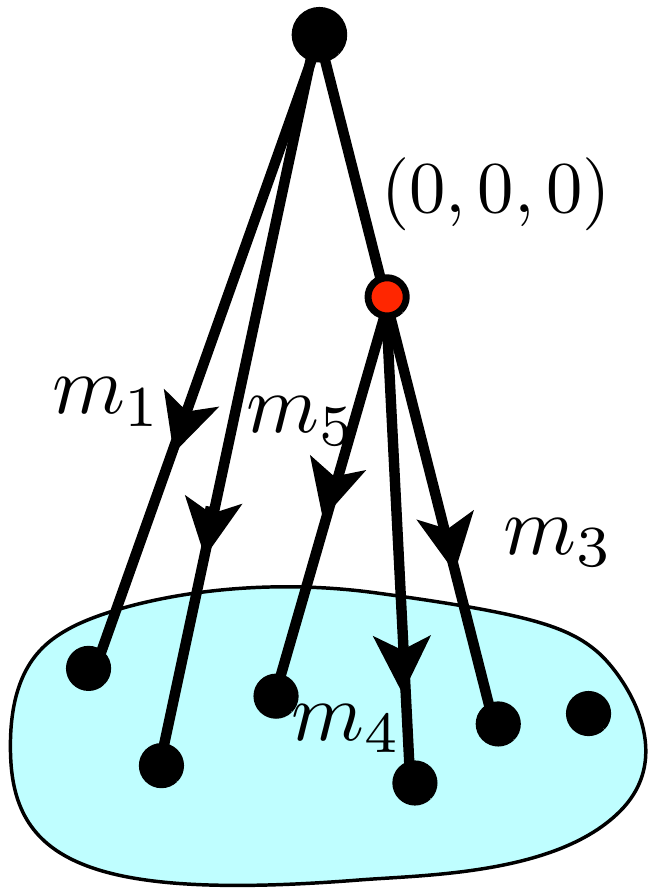} & \includegraphics[width=1in]{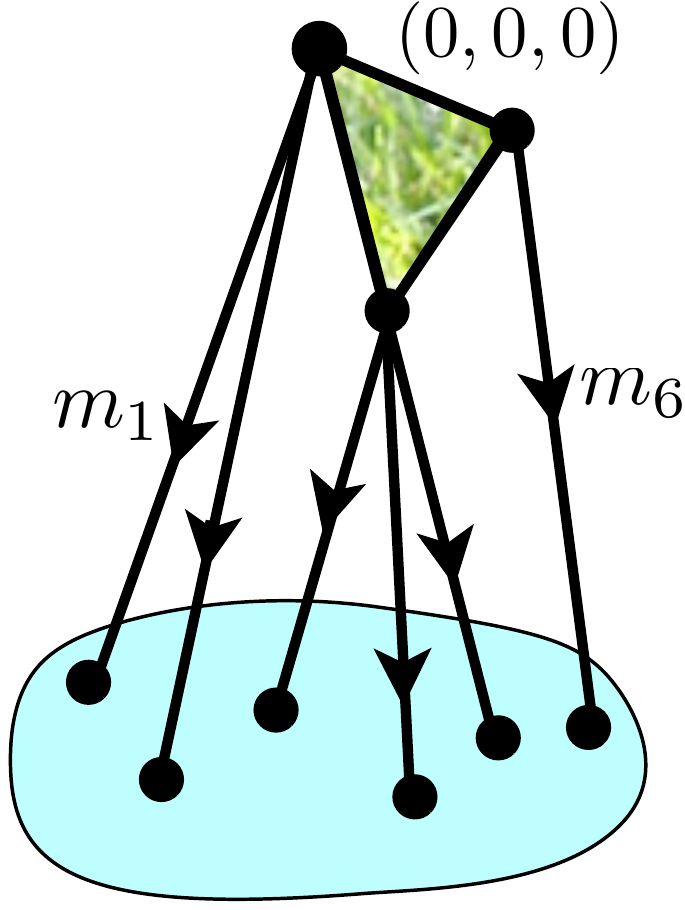} \\ 
& target & vertex addition & edge split & vertex addition \\
\hline
$K^*(6,1,0)$ & \includegraphics[width=1in]{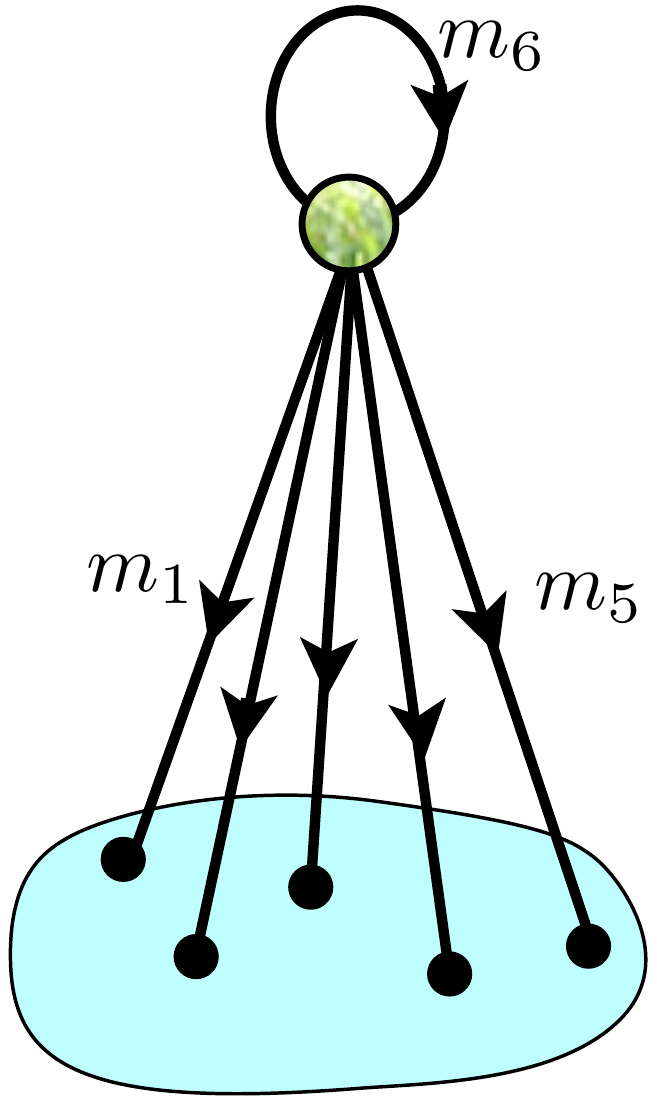} & \includegraphics[width=1in]{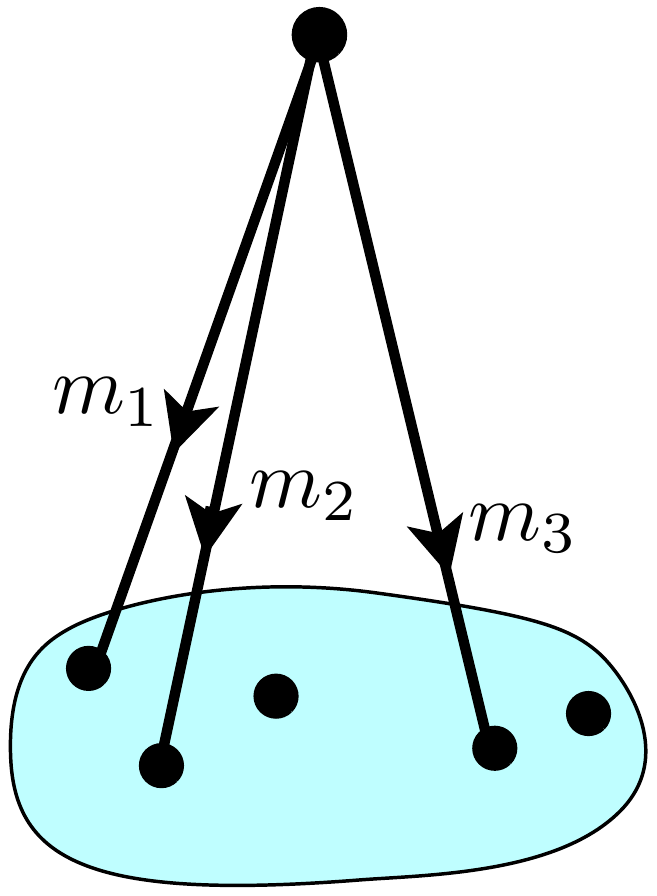} & \includegraphics[width=1in]{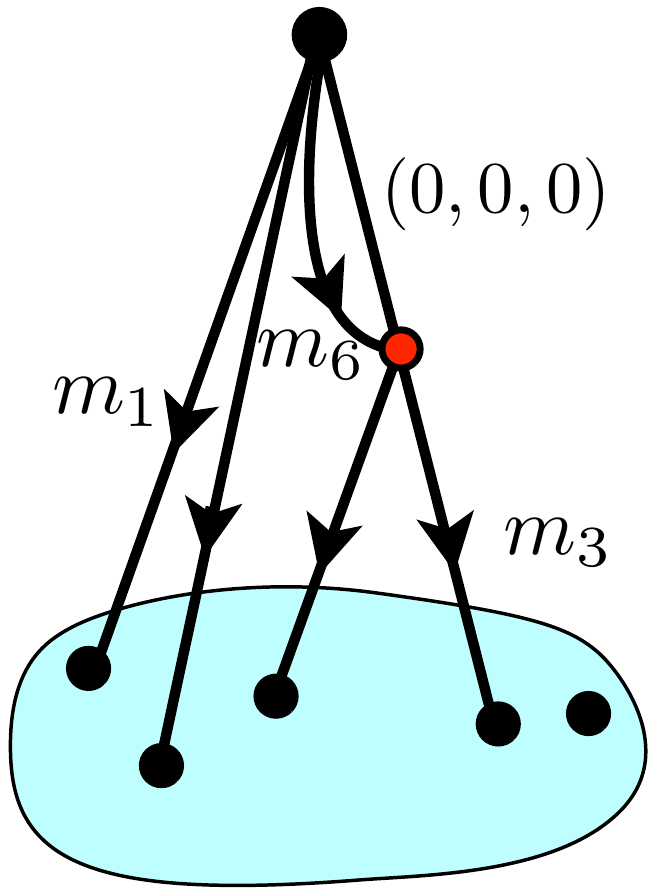} & \includegraphics[width=1in]{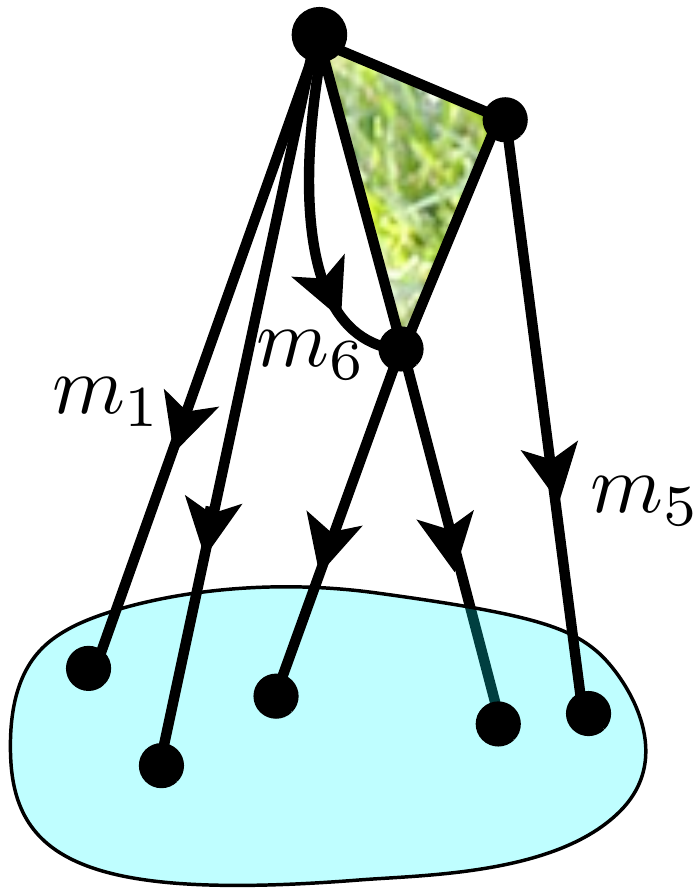} \\ 
& target & vertex addition & edge split, $m_6 \neq 0$ & vertex addition \\
\hline
$K^*(6,2,0)$ & \includegraphics[width=1in]{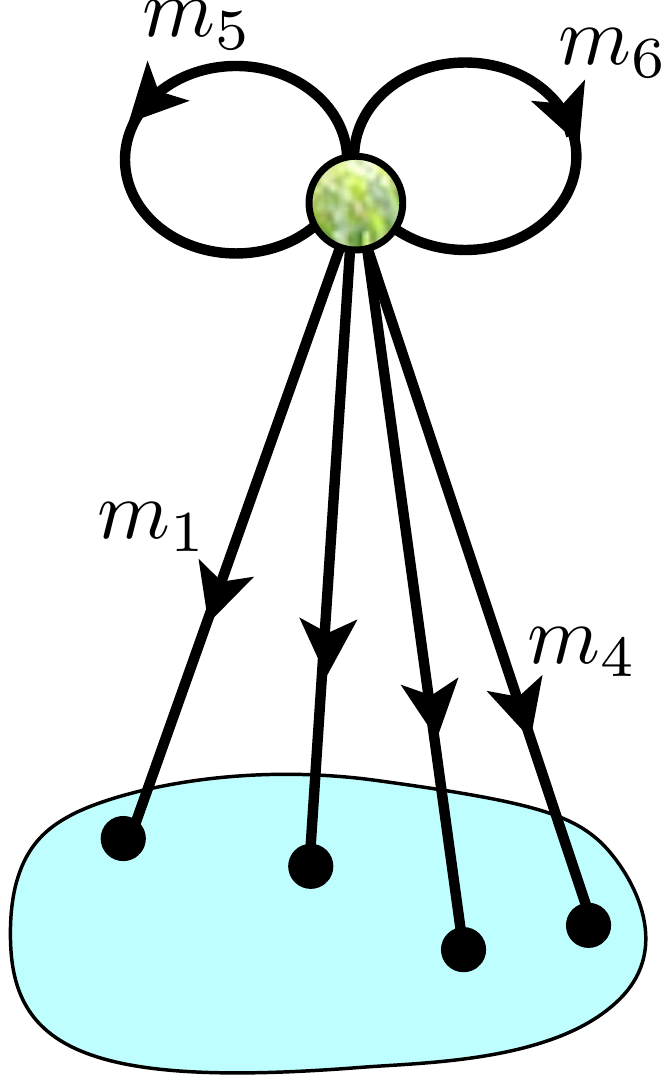} & \includegraphics[width=1in]{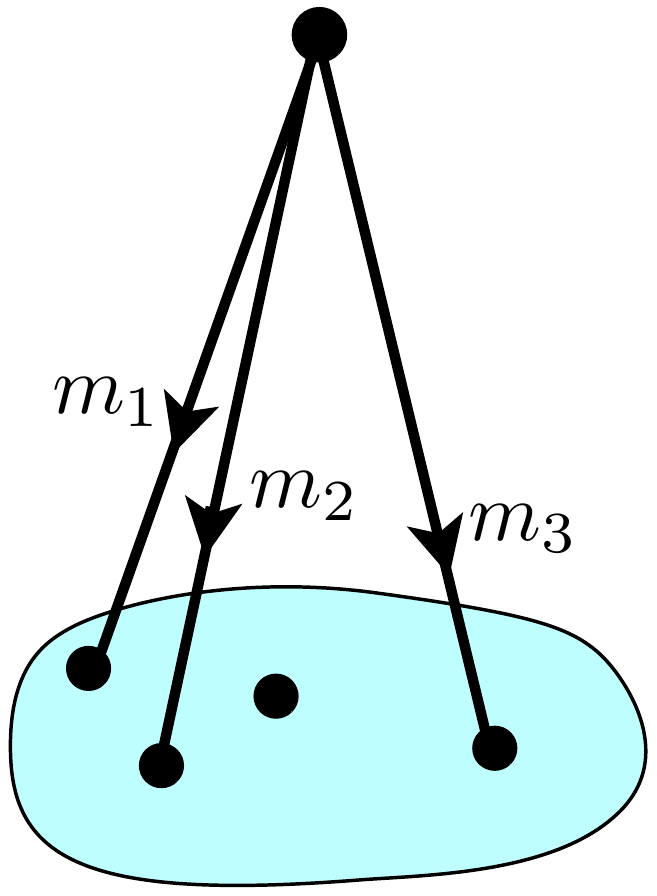} & \includegraphics[width=1in]{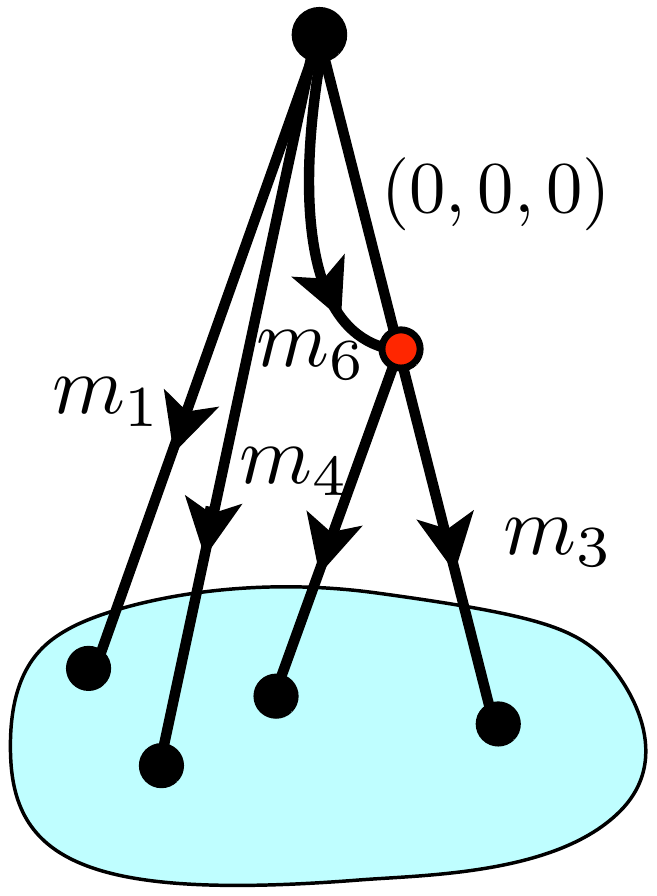} & \includegraphics[width=1in]{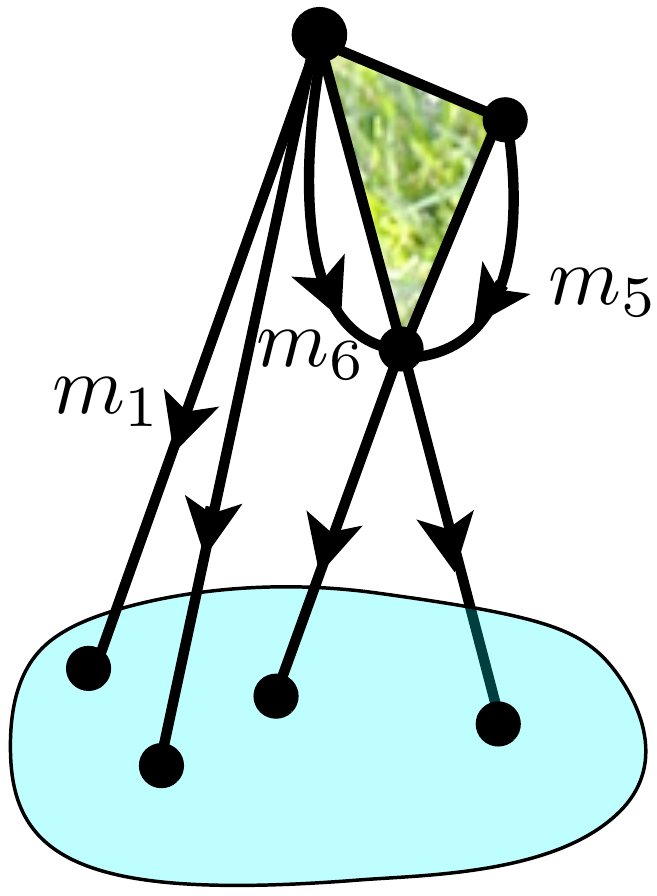} \\ 
& target & vertex addition & edge split, $m_6 \neq 0$ & vertex addition \\
\hline
$K^*(6,3,0)$ & \includegraphics[width=1in]{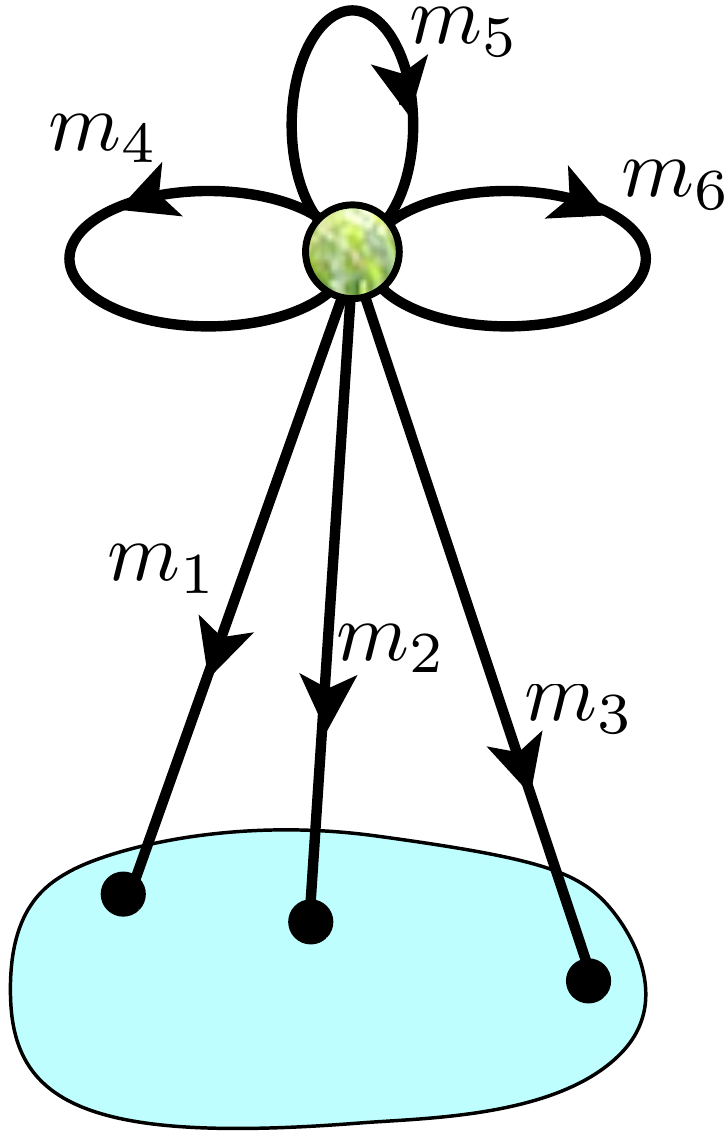} & \includegraphics[width=1in]{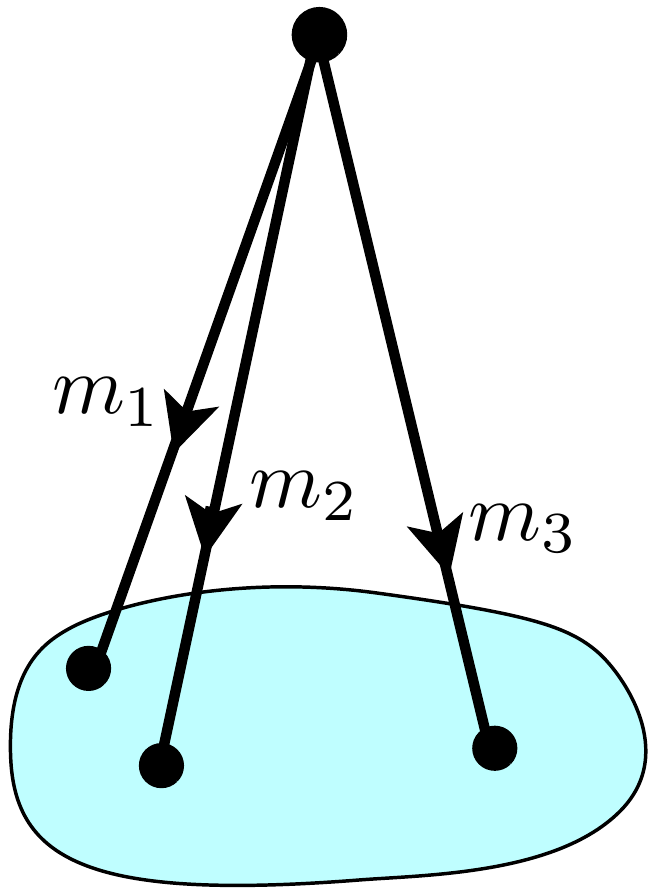} & \includegraphics[width=1in]{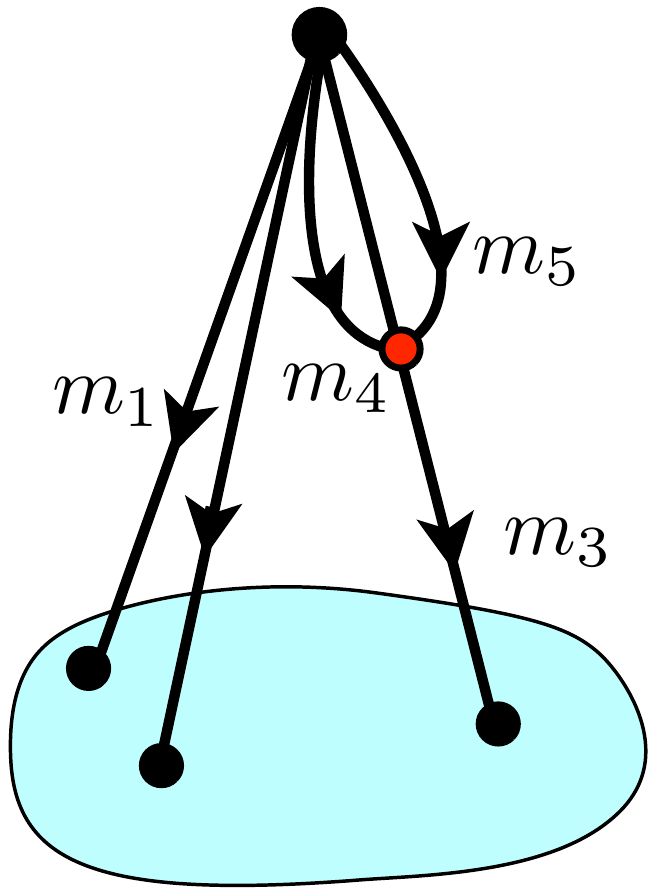} & \includegraphics[width=1in]{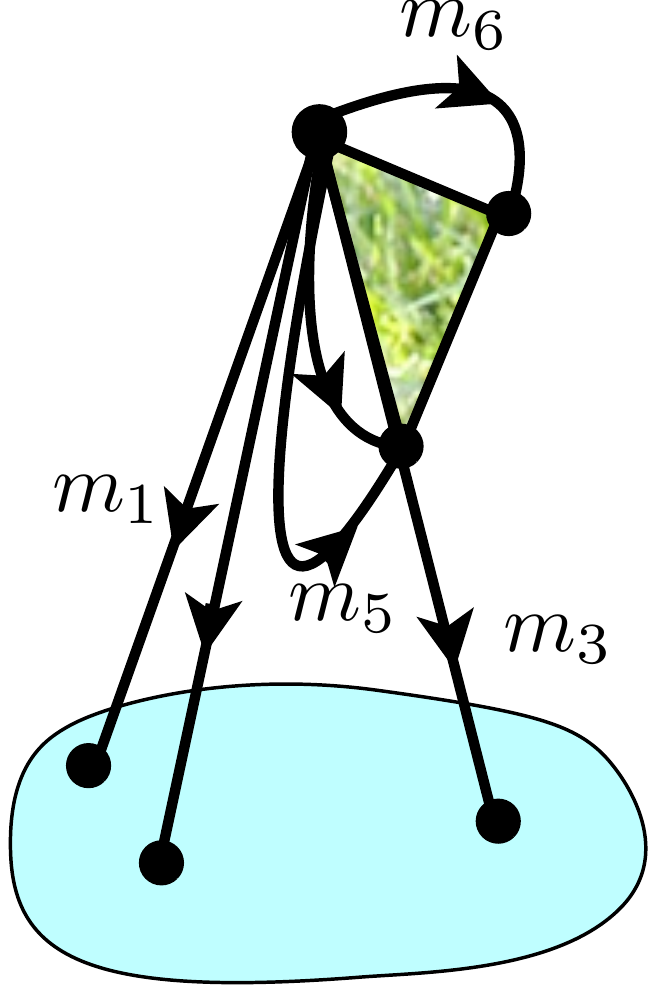} \\
& target & vertex addition & edge split,  & vertex addition \\
&  &  & $m_4 \neq m_5 \neq 0$ &  \\

\hline
\end{tabular}
\end{center}
\label{tab:bodyAddition}
\end{table}%

{\bf Step 2. $1 \leq j \leq 5$ (all other edge pinches)} 

Suppose we are splitting the edges $\{B_{\alpha_1}, B_{\beta_1}; m_1\}, \dots, \{B_{\alpha_j}, B_{\beta_j}; m_j\}$, $1 \leq j \leq 5$, and performing the edge split $K^*(6, n, j)$. We will think of this operation geometrically in terms of the derived periodic framework $(H^m, q^m)$. That is, we are splitting the following orbits of edges in $H^m$:
\[\{(B_{\alpha_1}, y), (B_{\beta_1}, z+m_1)\}, \dots, \{(B_{\alpha_j}, y), (B_{\beta_j}, z+m_j)\}, \ y, z \in \mathbb Z^3.\]

Performing a body addition (as in Step 1) to $\bbofw$, corresponds to adding a new orbit of bodies $(B_{\alpha_0}, z), z \in \mathbb Z^3$ to $(H^m, q^m)$. We claim that we can perform a body addition so that:
\begin{enumerate}
	\item the new body lies within a translate of the unit cell, and 
	\item $j$ of the edges adjacent to the new body lie along the lines of the edges we wish to split. 
\end{enumerate}
To see this, select $j$ generic points, one from along the line of each edge to be split. Because the set of generic points is dense, we may ensure that these $j$ points lie within one cell. Since the $j$ points are generic, we may define an isostatic bar-joint framework on these vertices (possibly adding other vertices if $j = 1, 2$), which corresponds to a body in $\bbofw$. This body lies within the unit cell, since all edges have gain $(0,0,0)$.

Let this body addition add $6-n$ additional edges that are not among those to be split, and $n-j$ loops, which are labeled as desired (see Figure \ref{fig:edgePinches}).  The $j$ edges which lie along the lines of the edges to be split need not have the same gains as the edges to be split, as illustrated. In fact, if we wish to create a graph with edges 
\[\{B_1, B_0; m_{a}\}, \{B_0, B_2; m_{b}\}\] we may always obtain this as an edge pinch of a graph with the edge $\{B_1, B_2; m_{a}+m_{b}\}$.

We complete the edge pinch by noting that the pair of edges in $\bbog$
\[\{B_{\alpha_1}, B_{\beta_1}; m_1\}, \{B_{\alpha_1}, B_0; m_{1,a}\}\]
is equivalent to the pair 
\[\{B_{0}, B_{\beta_1}; m_1-m_{1,a}\}, \{B_{\alpha_1}, B_0; m_{1,a}\},\]
by Proposition \ref{prop:triangleReplacement} applied to the corresponding orbits of edges in $(H^m, q^m)$.
Replacing every such pair in $\bbog$ (there are $j$ of them), we obtain the desired edge pinch, and the claim holds. 

\begin{figure}[htbp]
\begin{center}
\includegraphics[width=4in]{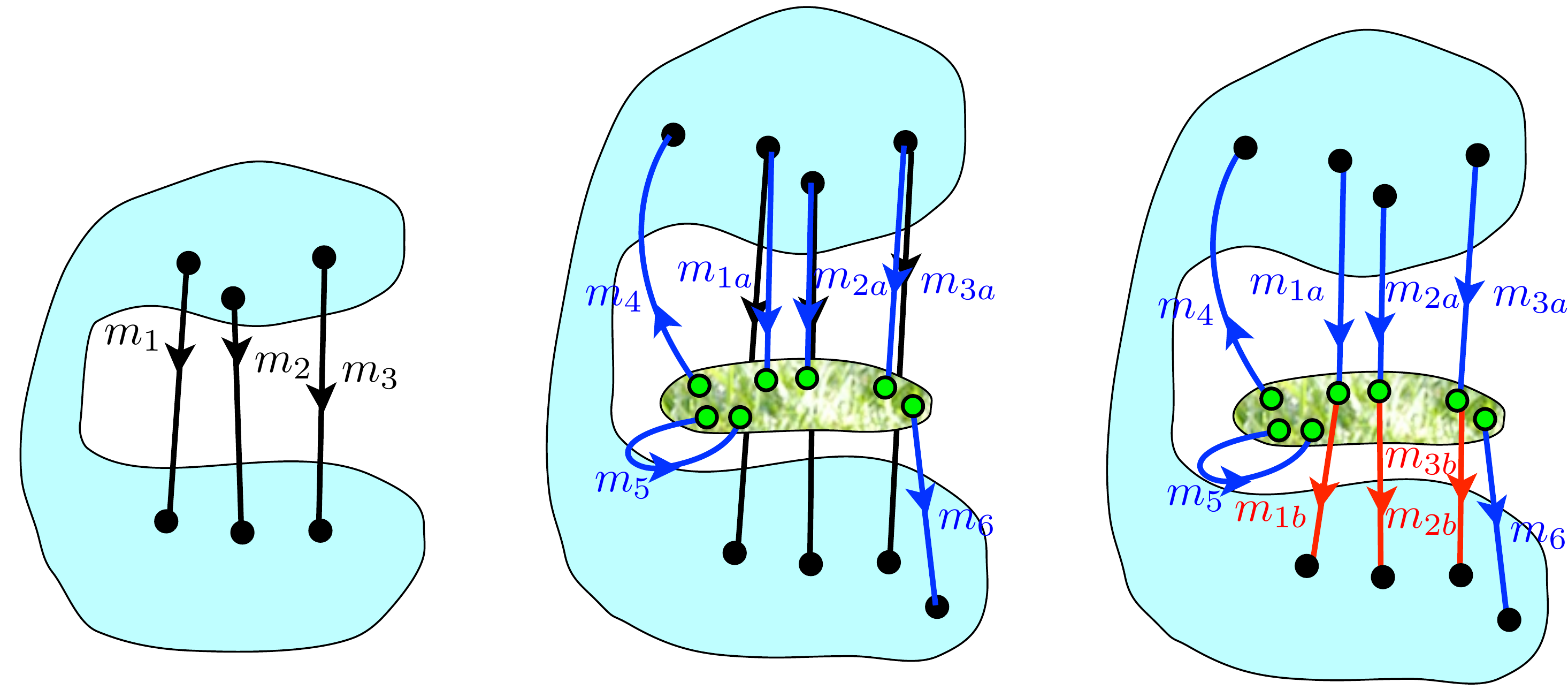}
\caption{Edge pinches $K^*(k,n,j)$, $j>0$.}
\label{fig:edgePinches}
\end{center}
\end{figure}

\end{proof}

%Combinatorial step
\subsection{Modified edge pinches generate the correct class of graphs}
\label{sec:combinatorialStep}

The main result of this section is as follows: 
\begin{thm}
Let $\bbog$ be a $[6,3]$-graph with the property that all non-empty subsets $Y \subset E$ of edges satisfy
		\[|Y| \leq 6|V(Y)| - 6 + \sum_{i=1}^{|\gs(Y)|}(3-i).\]
Then $\bbog$ can be obtained from $\langle P_3, m^*\rangle$, where $|\gs(P_3)|\geq 2$, by a sequence of gain-modified edge pinches. 
\label{thm:16mod}
\end{thm}

The proof of this result is based on the proof of Theorem \ref{thm:fekete}, as recorded in \cite{fekete}, with a few small modifications. 

By {\it splitting off} edges of the gain graph $\bbog$ we mean the operation of replacing the edges $e = \{u, v; m_e\}$ and $f = \{u, w; m_f\}$ with the new edge $ h = \{v, w; m_f - m_e\}$, which we call the {\it split edge}. The resulting graph we denote by $\langle G^{ef}, m^{ef} \rangle$, where $G^{ef} = (V, E-e-f+h)$, and $m^{ef}$ is the original gain assignment modified to include the new gain on the added edge $h$. 

\begin{thm}[analogous to \cite{fekete}, Theorem 2.2]
Let $\bbog$, $H =(V+s, E)$ be a $[6,3]$-graph satisfying (\ref{eq:gainSparse}), and let $n, j$ be integers such that $deg_H(s) = 6+n$, $i_H(s) = n-j$, where $j \leq n \leq 6$, and $n-j \leq 3$. Then we can split off $j$ pairs of edges so that after deleting $s$, the resulting graph is a $[6,3]$-graph which also satisfies (\ref{eq:gainSparse}). 
\label{thm:22mod}
\end{thm}
For brevity, we say that a $[6,3]$-graph which satisfies (\ref{eq:gainSparse}) is a {\it $[6,3]$-$\Tor^3$-graph}.
If $\bbog$, where $H=(V+s, E)$, is a $[6,3]$-$\Tor^3$-graph and $e, f$ are edges in $E\bbog$, then splitting off $e$ and $f$ is called {\it admissible} if $\langle H^{ef}, m^{ef} \rangle$ is also a $[6,3]$-$\Tor^3$-graph. Finally we define \[b_H(Y) = 6|V(Y)| - 6 + \sum_{i=1}^{|\gs(Y)|}(3-i) - |Y|.\]
Some authors have used a similar measure with the name ``deficiency", since $b_H$ describes how many edges can be added while maintaining independence. 
Note that a body-bar orbit graph $\bbog$ is a $[6,3]$-$\Tor^3$-graph if and only if $g_H(Y) \geq 0$ for all subsets of edges $Y \neq \emptyset$, $Y \subseteq E(H)$. 

To prove Theorem \ref{thm:22mod}, we make use of the following lemma:
\begin{lem}[analogous to \cite{fekete}, Lemma 2.3]
Let $\bbog$, $H=(V+s, E+\{e,f\})$ be a $[6,3]$-$\Tor^3$-graph, where $e = \{s, u; m_e\}$, $f=\{s, v; m_f\}$ are edges incident to $s$, $(u, v \in V)$. Then the pair $e, f$ is admissible if and only if $\nexists X \subseteq E$ with $u, v \in V(X)$, such that
\[b_H(X) = 0 \textrm{\ and \ } b_H(X+h) = -1,\]
where $h = \{u,v; m_f - m_e\}$ is the split edge. 
\label{lem:23mod5}
\end{lem}
We omit the proof. See Figure \ref{fig:admissible} for examples of admissible and non-admissible pairs of edges. 
\begin{figure}[htbp]
\begin{center}
\includegraphics[width=2in]{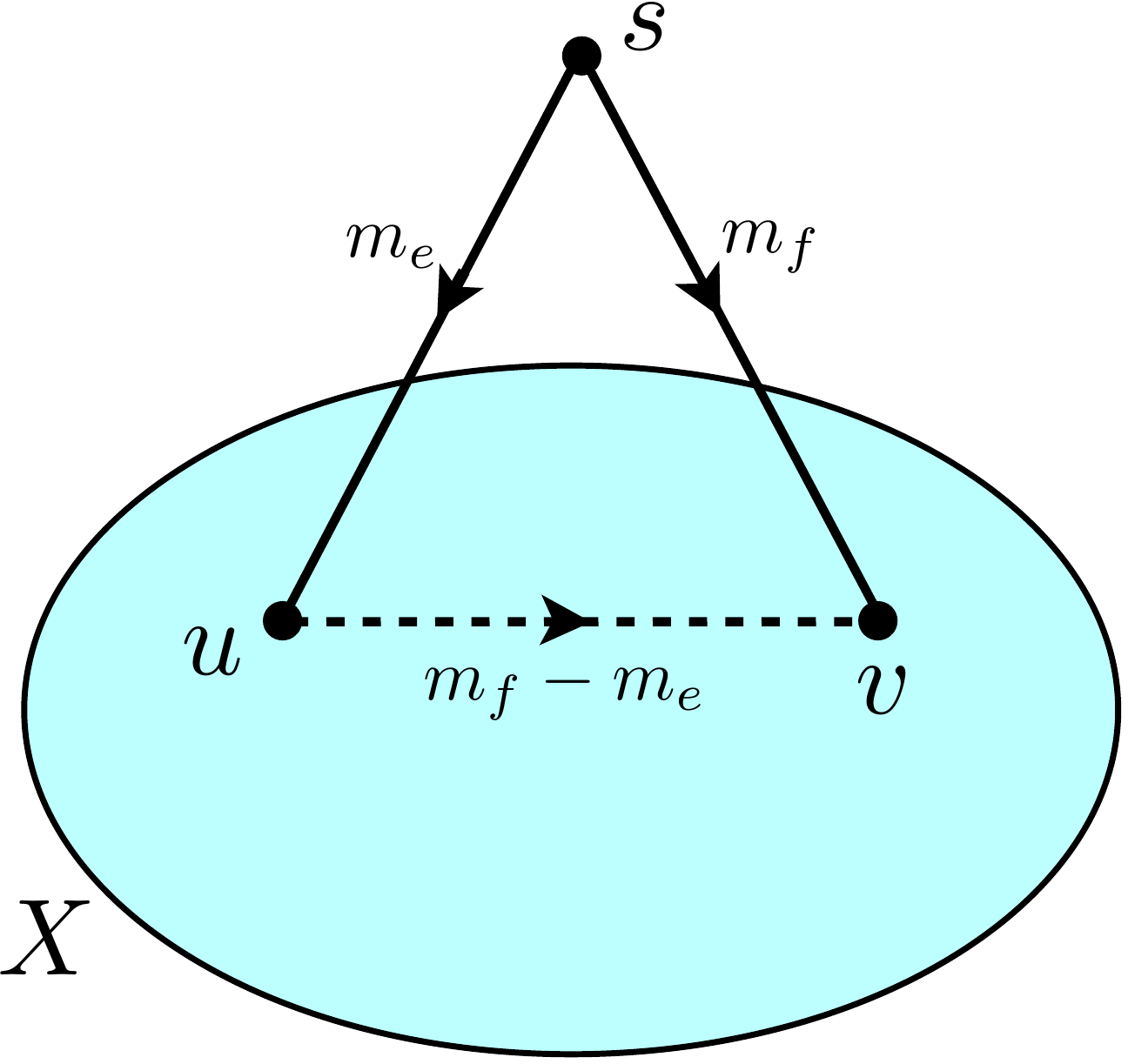}
\caption{Suppose that all edges in $X$ have trivial gains, and $|X| = 6|V(X)| - 6$ (i.e. $|\gs(X)| = 0$ and $b_H(X) = 0$). If $m_e = m_f$, then $b_H(X +h) = -1$, and the candidate pair (edge $e$ and $f$) is not admissible. However, if $m_e \neq m_f$ (for example, $m_e = (0,0,0)$, and $m_f = (1,0,0)$), then $|\gs(X+h)| = 1$, and therefore $b_H(X+h) = 1$, and the candidate pair is admissible.}
\label{fig:admissible}
\end{center}
\end{figure}

The following proposition is central to the proof of Theorem \ref{thm:main}.
\begin{prop}
Let $\bbog$ be a $[6,3]$-$\Tor^3$-graph, and let $X$ and $Y$ be subsets of edges of $H$. If $b_H(X) = b_H(Y) = 0$,  then $b_H(X\cup Y) = 0$ too. 
\label{prop:union}
\end{prop}

\begin{proof}
We first note that if $X \cap Y = \emptyset$, then the claim is trivial. We therefore assume that $X \cap Y \neq \emptyset$. We also remark that 
\[\sum_{i=1}^2(3-i) = \sum_{i=1}^3(3-i) = 3,\]
and therefore we need not consider the case $|\gs(Y)| =2$ separately from $|\gs(Y)| = 3$. We thus take cases as follows:
\begin{table}[htdp]
%\caption{default}
\begin{center}
\begin{tabular}{|c|c|c|}
\hline
Case & $|\gs(X)|$ & $|\gs(Y)|$ \\
\hline
a & 0 & 0 \\
b & 1 & 1 \\
c & 2 & 2 \\
d & 0 & 1 \\
e & 0 & 2 \\
f & 1 & 2\\
\hline
\end{tabular}
\end{center}
\label{default}
\end{table}%

\noindent {\bf Case a:} $|\gs(X)| = |\gs(Y)| = 0$.\\
Since $b_H(X) = b_H(Y) = 0$, we have $|X| = 6|V(X)| - 6$ and $|Y| = 6|V(Y)| - 6$. It must also be the case that $|\gs(X \cap Y)| = 0$, and therefore 
\begin{equation} 
|X \cap Y| \leq 6|V(X \cap Y)| - 6. 
\label{eq:caseA1}
\end{equation}
We have 
\begin{eqnarray}
	|X \cup Y| + |X \cap Y| & = & |X| + |Y|  \nonumber \\
					& = & (6|V(X)| - 6) + (6|V(Y) - 6) \nonumber \\
					& = & 6|V(X \cup Y)| + 6|V(X \cap Y)| - 12. 
\label{eq:caseA2}
\end{eqnarray}
Together, (\ref{eq:caseA1}) and (\ref{eq:caseA2}) imply that 
\[6|V(X \cup Y)|-6 \leq |X \cup Y|.\]
Since $\bbog$ is a $[6,3]$-$\Tor^3$-graph, $6|V(X \cup Y)|-6 < |X \cup Y|$ if and only if $|\gs(X \cup Y)|>0$. We will show that $|\gs(X \cup Y)| = 0$, and therefore $6|V(X \cup Y)|-6 = |X \cup Y|$, and $b_H(X \cup Y) = 0$, as desired. 

Suppose toward a contradiction that $|\gs(X \cup Y)| = 1$. Then $|X \cup Y| \leq 6|V(X \cup Y)| - 4$, and by (\ref{eq:caseA2}), 
\begin{eqnarray}6|V(X \cap Y)| - 8 \leq |X \cap Y|.
\label{eq:caseA3}
\end{eqnarray}
For a cycle to be created with non-trivial gain, the intersection $V(X \cap Y)$ must contain at least two vertices, since any such cycle must contain edges from both $X$ and $Y$. It follows that $4 \leq |X \cap Y|$, and the intersection is non-trivial. Furthermore, the intersection must be connected. (Suppose that $X\cap Y$ is disconnected, into say $E_1$ and $E_2$. We must have $|E_i| \leq 6|V(E_i)| - 6$, and therefore $|E_1| + |E_2| \leq 6(|V(E_1)| + |V(E_2)|)-12$, but this contradicts (\ref{eq:caseA3}).)

Suppose that a non-zero cycle is created containing $a, b \in V(X\cap Y)$. In the simplest case, the cycle consists of a single sequence of edges from $X$, followed by a single sequence of edges from $Y$ (see Figure \ref{fig:caseA1}). We will treat other cases separately. Suppose that the cycle is as follows $C = a, C_X, b, C_Y, a$, where $C_X$ is a path of vertices and edges in $X$, and $C_Y$ is a path of vertices and edges in $Y$. Suppose further that the path $C_X$ has gain $m_X$, and the path $C_Y$ has gain $m_Y$. Since the intersection $(V(X\cap Y), X \cap Y)$ is connected, there is a path from $a$ to $b$ within the intersection. Suppose this path has gain $m_{int}$. Then we may write the cycle $a, C_X, b, C_Y, a$ as $a, C_X, b, -C_{int}, a, C_{int}, b, C_Y, a$. 
But note that the cycle $a, C_X, b, -C_{int}, a$ is completely contained in the edge set $X$, hence must have net gain zero. Similarly, $a, C_{int}, b, C_Y, a$ is completely contained in the edge set $Y$, hence must have net gain zero. It follows that the net gain of the cycle $a, C_X, b, C_Y, a$ is:
\[m_X + m_Y = (m_X - m_{int}) + (m_{int} +m_Y) = 0 + 0 = 0.\]

\begin{figure}[htbp]
\begin{center}
\includegraphics[width=3in]{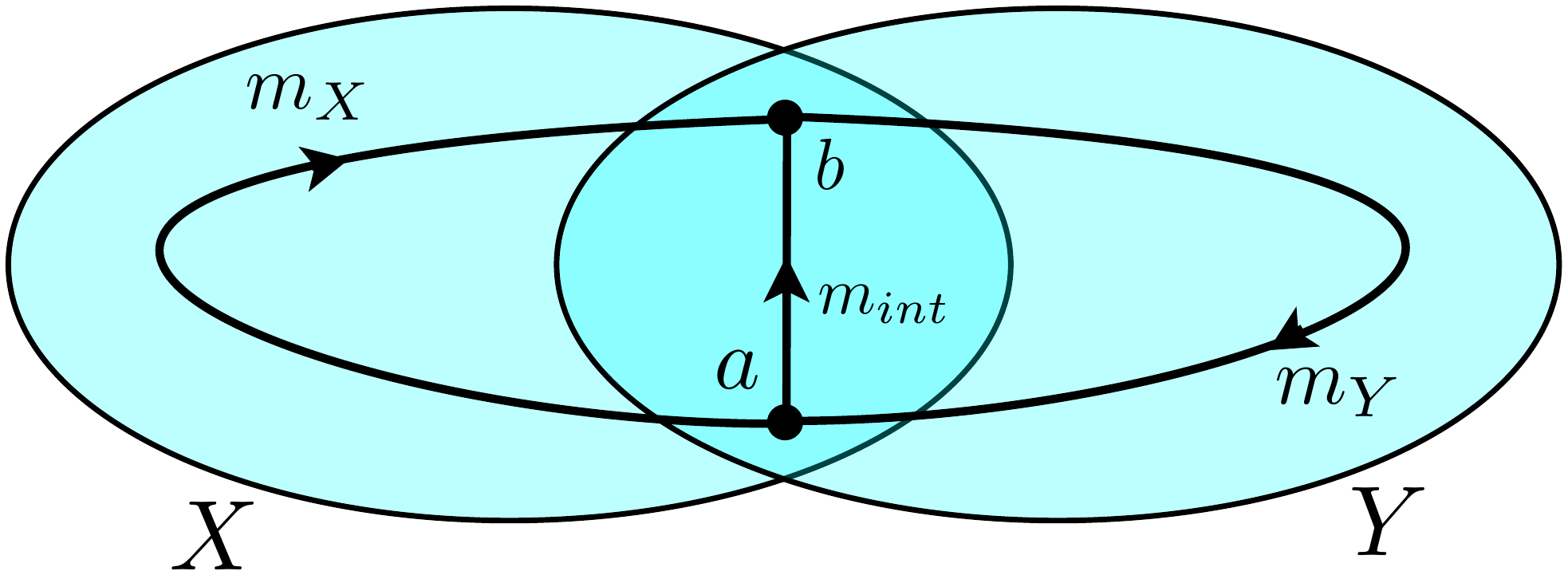}
\caption{A simple cycle with net gain $m_X + m_Y$.}
\label{fig:caseA1}
\end{center}
\end{figure}

It is a small extension to include cases where the cycle alternates between edges in $X$ and edges in $Y$ more than once (see Figure \ref{fig:caseA2}). The cycle $a, C_1, c_1, C_2, c_2, C_3, b, C_Y, a$ can also be written as the sum of cycles in $X$ and cycles in $Y$: 
\[a \underbrace{\xrightarrow{m_1} c_1 \rightarrow}_{0} a \underbrace{\rightarrow c1 \xrightarrow{m_2} c_2 \rightarrow}_{0} a \underbrace{\rightarrow c_2 \xrightarrow{m_3} b \rightarrow}_{0} a \underbrace{\rightarrow b \xrightarrow{m_Y}}_{0} a.\]
Here the first and third closed cycles at $a$ have net gain $0$ since they are completely contained in $X$, and the second and fourth closed cycles at $a$ have net gain $0$ because they are completely contained in $Y$. We have shown that it is not possible to have any cycles in $X\cup Y$ with non-zero gain, hence we must have $|\gs(X \cup Y)|=0$, and therefore $|X\cup Y| = 6|V(X \cup Y)| - 6$ and $b_H(X \cup Y) = 0$, as desired.\\

\begin{figure}[htbp]
\begin{center}
\includegraphics[width=3in]{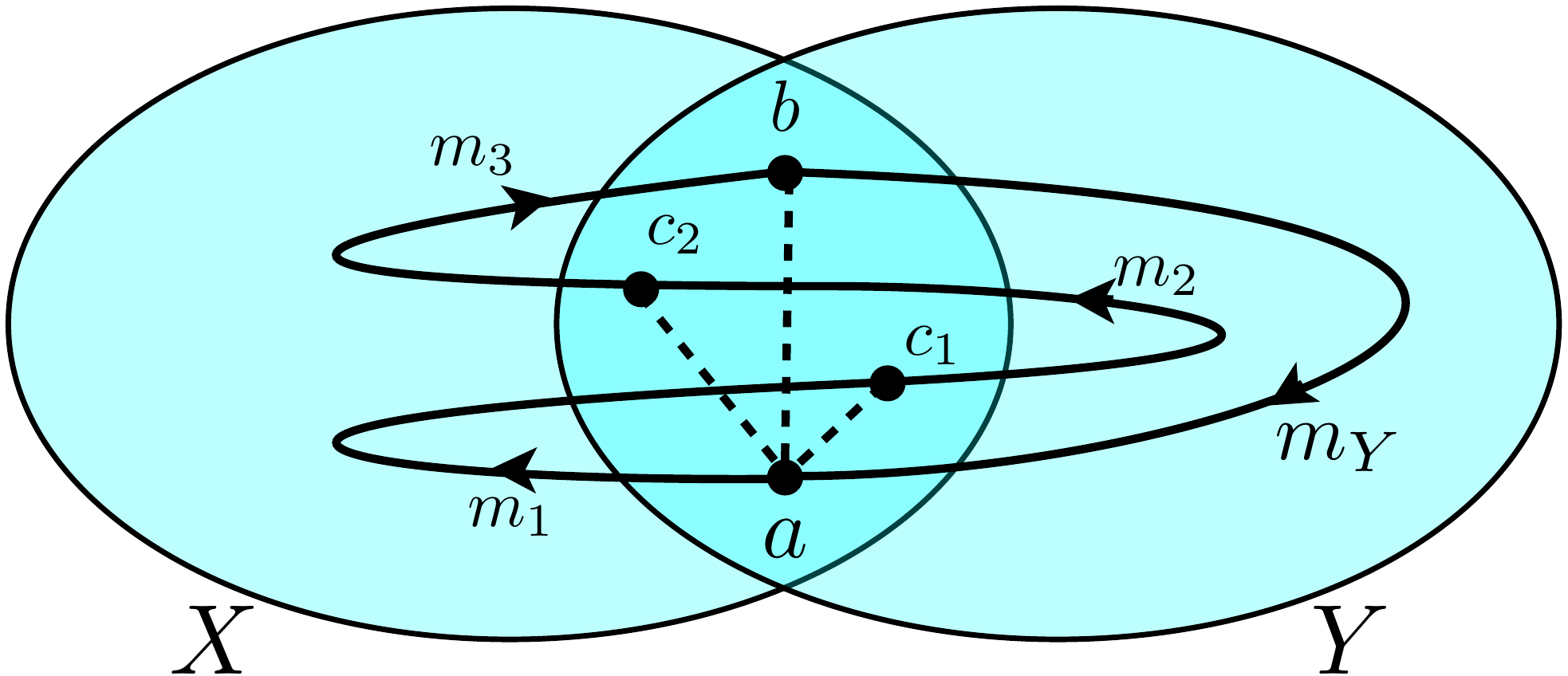}
\caption{A cycle containing edges of both edge sets $X$ and $Y$, but which alternates between paths in $X$ and paths in $Y$.}
\label{fig:caseA2}
\end{center}
\end{figure}

\noindent {\bf Case b:} $|\gs(X)| = |\gs(Y)| = 1$.\\
Since $b_H(X) = b_H(Y) = 0$, we have $|X| = 6|V(X)| - 4$ and $|Y| = 6|V(Y)| - 4$. Then 
\begin{equation} 
6|V(X \cup Y)| + 6|V(X \cap Y)| - 8 = |X \cap Y| + |X \cup Y|.
\label{eq:caseB1}
\end{equation}
Now suppose that $|\gs(X \cap Y)| = 0$. Then $|X \cap Y| \leq 6|V(X \cap Y)| - 6$. Then by (\ref{eq:caseB1}), $6|V(X \cup Y)| - 2 \leq |X \cup Y|$, which is a contradiction, since no subgraph of a $[6,3]$-tight graph may satisfy this. Therefore, it must be the case that $|\gs(X \cap Y)| = 1$ (it cannot be greater than $1$, since it is a subset of both $X$ and $Y$). Hence 
\[|X \cap Y| \leq 6|V(X \cap Y)| -4, \textrm{ \ and}\]
\[|X \cup Y| \geq 6|V(X \cup Y)| -4.\]
Is it possible that $|\gs(X \cup Y)|=2$ (or $3$)? We claim that it is not possible.

First note that since $|\gs(X \cap Y)| = 1$, we may write $\gs(X \cap Y) = \langle m_0 \rangle = \{km_0 | k \in \mathbb Z\}$. That is, the gain space of $X \cap Y$ is generated by the gain $m_0 \in \mathbb Z^3$. However, since $X \cap Y$ is a subset of both $X$ and $Y$, each of which also have gain spaces of dimension 1, it follows that $\gs(X) = \gs(Y) = \langle m_0 \rangle$ too. 

The rest of the argument is similar to that in Case a. Reconsidering Figure \ref{fig:caseA2} for this case, we obtain the cycle:  
\[a \underbrace{\xrightarrow{m_1} c_1 \rightarrow}_{k_1m_0} a \underbrace{\rightarrow c1 \xrightarrow{m_2} c_2 \rightarrow}_{k_2m_0} a \underbrace{\rightarrow c_2 \xrightarrow{m_3} b \rightarrow}_{k_3m_0} a \underbrace{\rightarrow b \xrightarrow{m_Y}}_{k_Ym_0} a,\]
where $k_i, k_Y \in \mathbb Z$. Hence the net gain on this cycle is $(k_1 + k_2 + k_3 + k_Y)m_0$, which is clearly contained in $\langle m_0 \rangle$. It is not possible, therefore that $|\gs(X \cup Y)| > 1$. It follows that $|\gs(X \cup Y)| = 1$, and thus $b_H(X \cup Y) = 0$, as desired. \\

\noindent {\bf Case c:} $|\gs(X)| = |\gs(Y)|= 2$.\\
Since $b_H(X) = b_H(Y) = 0$, we have $|X| = 6|V(X)| - 3$ and $|Y| = 6|V(Y)| - 3$. In this case, the fact that $b_H(X \cup Y) = b_H(X \cap Y)$ follows directly from a lemma of Fekete and Szeg\H{o}:
\begin{lem}[\cite{fekete}, Lemma 2.3]
If $H$ is a $[6,3]$-tight graph, then $b_H(X) = b_H(Y) = 0$, and $X \cap Y \neq \emptyset$ implies that $b_H(X \cup Y) = b_H(X \cap Y) = 0$. 
\end{lem}

\noindent {\bf Case d:} $|\gs(X)| = 0, |\gs(Y)| = 1$.\\
In this case, $|X| = 6|V(X)| - 6$, and $|Y| = 6|V(Y) - 4$. It follows that $|\gs(X \cap Y)| \leq |\gs(X)|= 0$, and $|\gs(X \cup Y)| \geq |\gs(Y)| = 1$. Hence $|X\cap Y| \leq 6|V(X\cap Y)| - 6$. 
We have
\begin{equation} 
6|V(X \cup Y)| + 6|V(X \cap Y)| - 10 = |X \cap Y| + |X \cup Y|.
\label{eq:caseD1}
\end{equation}
We claim that $|\gs(X\cup Y)| = 1$. Suppose toward a contradiction that $|\gs(X \cup Y)| = 2$. That is, a new cycle is created with a gain that is not generated by $\langle m_0 \rangle = \gs(Y)$. Then $|V(X \cap Y)| \geq 2$ (for a new cycle to be created), and the intersection is connected. Again we argue along the same lines as Cases a and b. Considering Figure \ref{fig:caseA2}, we write the cycle as:  
\[a \underbrace{\xrightarrow{m_1} c_1 \rightarrow}_{0} a \underbrace{\rightarrow c1 \xrightarrow{m_2} c_2 \rightarrow}_{k_2m_0} a \underbrace{\rightarrow c_2 \xrightarrow{m_3} b \rightarrow}_{0} a \underbrace{\rightarrow b \xrightarrow{m_Y}}_{k_Ym_0} a.\]
Hence the net gain on this cycle is $(k_2 + k_Y)m_0$, which is clearly contained in $\langle m_0 \rangle= \gs(Y)$. It is not possible, therefore that $|\gs(X \cup Y)| > 1$. It follows that $|\gs(X \cup Y)| = 1$, and thus $b_H(X \cup Y) = 0$, as desired. \\

\noindent {\bf Case e:} $|\gs(X)| = 0, |\gs(Y)| = 2$ (or 3).\\
In this case, $|X| = 6|V(X)| - 6$, and $|Y| = 6|V(Y) - 3$. It follows that $|\gs(X \cap Y)| \leq |\gs(X)|= 0$, and $|\gs(X \cup Y)| \geq |\gs(Y)| = 2$ (or 3). Hence $|X\cap Y| \leq 6|V(X\cap Y)| - 6$. 
But since 
\begin{equation} 
6|V(X \cup Y)| + 6|V(X \cap Y)| - 9 = |X \cap Y| + |X \cup Y|,
\label{eq:caseE1}
\end{equation}
it follows that $|X\cap Y| = 6|V(X\cap Y)| - 6$, and $|X\cup Y| = 6|V(X\cup Y)|-3$. Hence $b_H(X \cap Y) = b_H(X\cup Y) = 0$, as desired.\\

\noindent {\bf Case f:} $|\gs(X)| = 1, |\gs(Y)| = 2$ (or 3).\\
In this case, $|X| = 6|V(X)| - 4$, and $|Y| = 6|V(Y) - 3$. It follows that $|\gs(X \cap Y)| \leq |\gs(X)|= 1$, and $|\gs(X \cup Y)| \geq |\gs(Y)| = 2$ (or 3). Hence $|X\cap Y| \leq 6|V(X\cap Y)| - 4$. Again, since
\begin{equation} 
6|V(X \cup Y)| + 6|V(X \cap Y)| - 7 = |X \cap Y| + |X \cup Y|,
\label{eq:caseF1}
\end{equation}
it follows that $|X\cap Y| = 6|V(X\cap Y)| - 4$, and $|X\cup Y| = 6|V(X\cup Y)|-3$. Hence $b_H(X \cap Y) = b_H(X\cup Y) = 0$, as desired.\\
\end{proof}

We also require the following proposition, which essentially says that if subsets $X,Y\subseteq E(H)$ prevent the splitting off of either of two pairs of edges, then so does the subset $X \cup Y$. 
\begin{prop}
Let $\bbog$, $H=(V(H)+s, E(H)+\{e,f,g\})$ be a $[6,3]$-$\Tor^3$-graph, where $e = \{s, u; m_e\}$, $f=\{s, v; m_f\}$, $g=\{s,w; m_g\}$ are edges incident to $s$. Let $X$ and $Y$ be subsets of $E(H)$ such that $u, v \in V(X)$, and $v, w \in V(Y)$. Suppose $h_1 = \{u, v; m_f-m_e\}$ and $h_2 = \{v, w; m_g - m_f\}$ (See Figure \ref{fig:neitherEdge}). If 
\[b_H(X+h_1) = b_H(Y+h_2) = -1,\] then $b_H(X \cup Y+h_1)= b_H(X \cup Y+h_2) = -1$ too.
\label{prop:unionNeg}
\end{prop}

\begin{figure}[htbp]
\begin{center}
\includegraphics[width=2.4in]{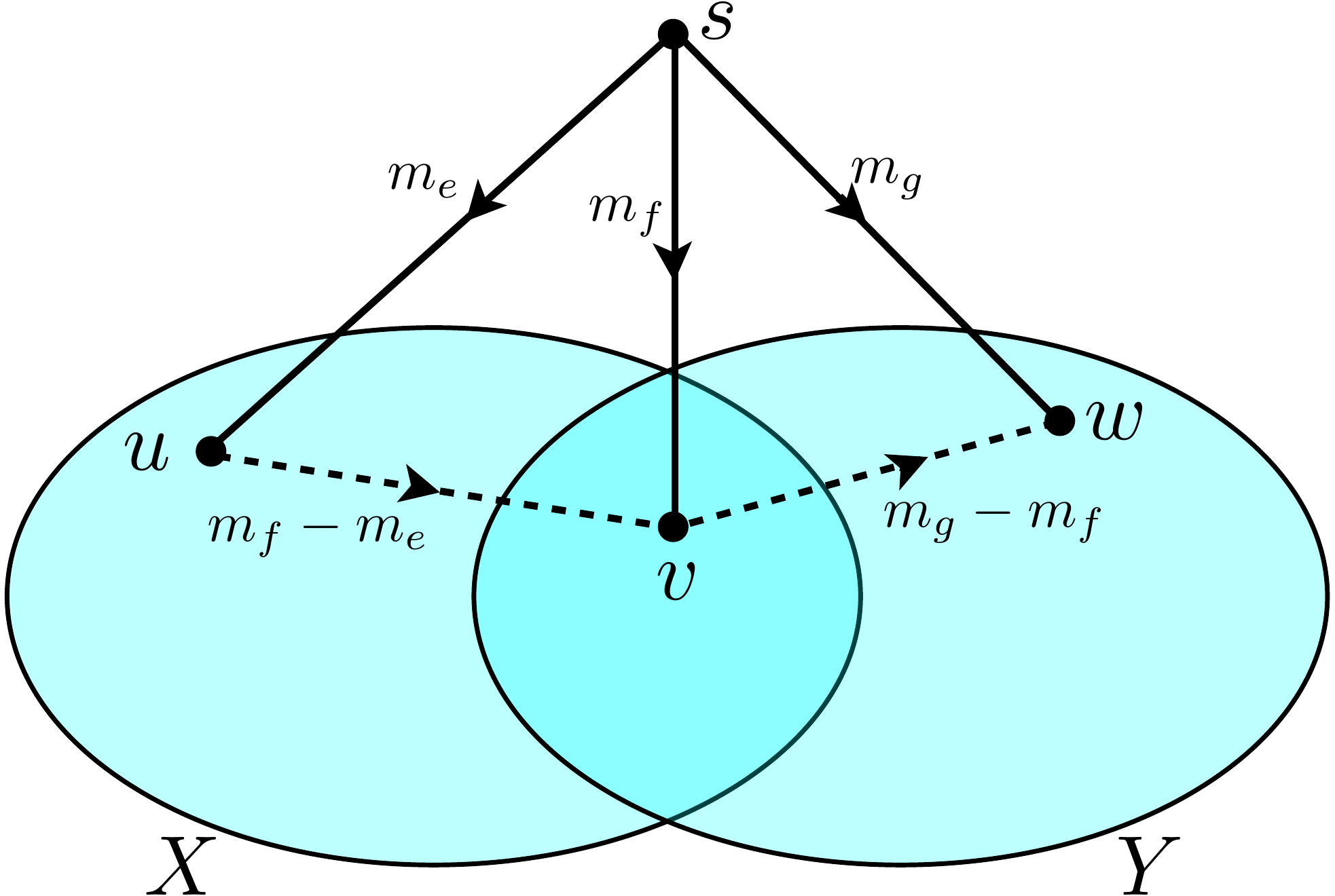}
\caption{Edge sets $X$ and $Y$ prevent the addition of either of the edges $\{u, v; m_f - m_e\}$ or $\{v, w; m_g - m_f\}$. See Proposition \ref{prop:unionNeg}. }
\label{fig:neitherEdge}
\end{center}
\end{figure}

\begin{proof}
The proof of this proposition proceeds along the lines of the proof of Proposition \ref{prop:union}. In particular, we use Cases a -- f. We will prove one case in detail, the remaining cases are easily checked using the same methods.  \\

\noindent {\bf Case d:} $|\gs(X)| = 0, |\gs(Y)| = 1$.\\
What we must confirm is that adding either of the edges $h_1$ or $h_2$ to $X\cup Y$ results in an overbraced subgraph of $H$. In other words, we prove that adding $h_1$ or $h_2$ cannot increase the dimension of the gain space of $X \cup Y$. 

\begin{figure}[htbp]
\begin{center}
\includegraphics[width=2.4in]{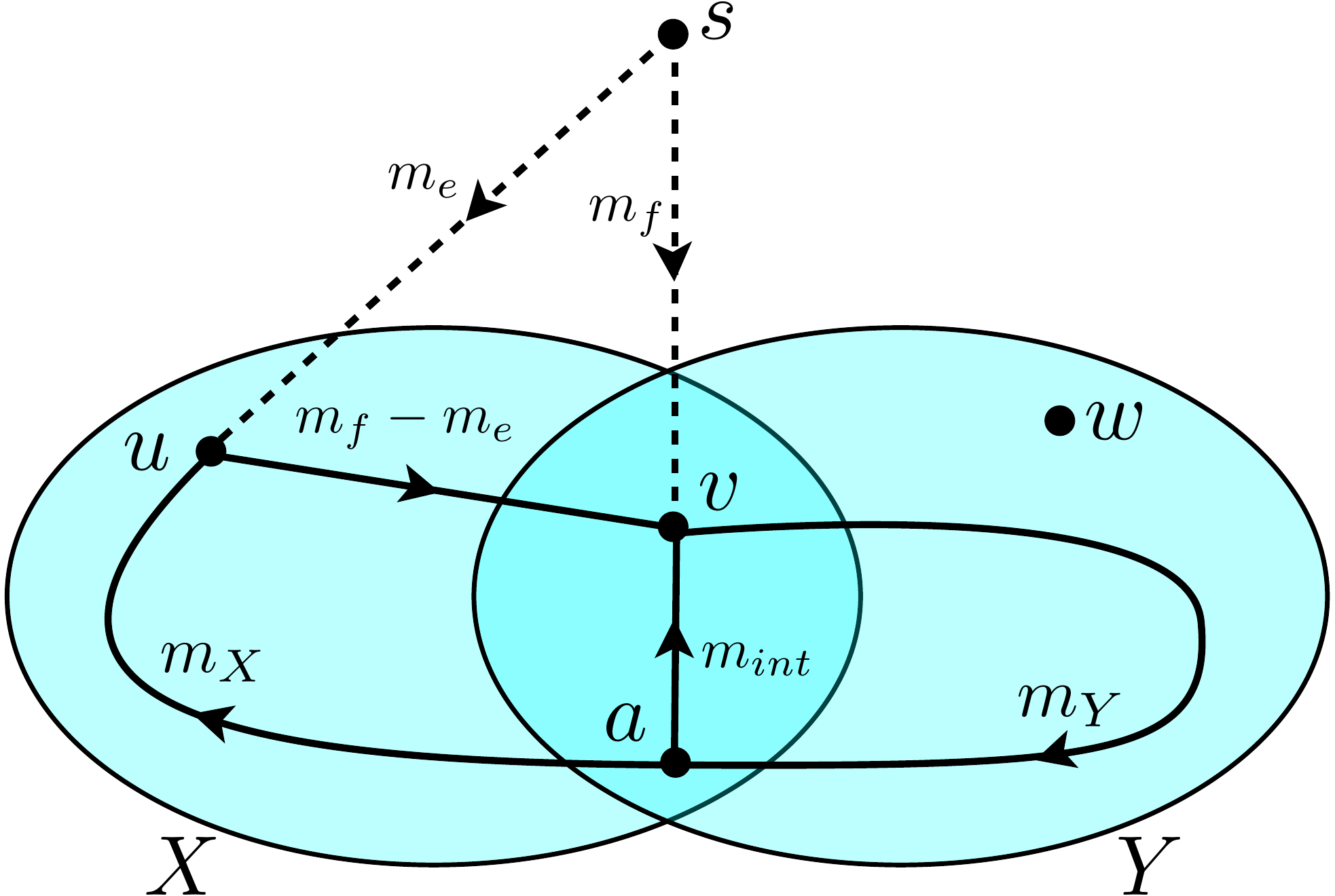}
\caption{Suppose the edge $h_1 = \{u, v; m_f - m_e\}$ participates in a cycle which raises the rank of $X \cup Y +h_1$ (proof of Proposition \ref{prop:unionNeg}).}
\label{fig:addedEdgeCycle}
\end{center}
\end{figure}

From the proof of Proposition \ref{prop:union} we know that $|\gs(X \cup Y)| = 1$. Suppose toward a contradiction, that $|\gs(X \cup Y + h_1)| = 2$. Then it must be the case that the added edge $h_1$ participates in a cycle whose net gain is not contained in $\gs(X \cup Y) = \langle m_0 \rangle$ (see Figure \ref{fig:addedEdgeCycle}). Furthermore, such a cycle must contain edges from both edge sets $X$ and $Y$, and thus $|V(X \cap Y)|\geq 2$.  However, from the proof of Proposition \ref{prop:union} we know that the intersection $X \cap Y$ is connected, and therefore we may break the cycle up into the sum of simple cycles, each of which is contained only in $X$ or only in $Y$ (see Figure \ref{fig:addedEdgeCycle}). But we easily see that such a cycle has net gain $km_0$, $k \in \mathbb Z$, which is the contradiction. 

A similar argument shows that $|\gs(X \cup Y + h_2)| = 1$ too, and therefore
\[b_H(X \cup Y+h_1)= b_H(X \cup Y+h_2) = -1\]
as desired. 
\end{proof}

With Propositions \ref{prop:union} and \ref{prop:unionNeg} in place, the proofs of Theorems \ref{thm:22mod} and \ref{thm:16mod} follow verbatim from the work of Fekete and Szeg\H{o} (see \cite{fekete}, Theorems 2.2 and 1.6 respectively. Lemma \ref{lem:23mod5} and Propositions \ref{prop:union} and \ref{prop:unionNeg} are analogous to Lemma 2.3 in \cite{fekete}). We sketch the principal ideas of both proofs, identifying the areas which require modification. \\

\noindent {\it Proof of Theorem \ref{thm:22mod} (sketch).}\\
First note that if $j = 0$, then $G-s$ must be a \bbgood. So we assume that $j \geq 1$. Toward a contradiction, suppose that we cannot split off $j$ edges so that the resulting graph is a \bbgood. We split off as many edges as possible, say $p<j$ to obtain the graph $\langle H', m' \rangle$. Let $e_1 = \{s, v_1; m_1\}, \dots, e_{\alpha} = \{s, v_{\alpha}; m_{\alpha}\}$ be the non-loop edges incident to $s$ in $\langle H', m' \rangle$. By Lemma \ref{lem:23mod5} we know that for each pair of vertices $v_{\nu}, v_{\mu}$, $1\leq \nu < \mu \leq \alpha$, there exists an edge set $X_{\nu \mu} \subset E(H')$ such that  
\begin{enumerate}[a)]
\item $v_{\nu}, v_{\mu} \in V(X_{\nu \mu})$
\item $b_H(X_{\nu \mu}) = 0$, and 
\item $b_H(X_{\nu \mu} + \{v_{\nu}, v_{\mu}; m_{\mu} - m_{\nu} \}) = -1$.
\end{enumerate}
Using Propositions \ref{prop:union} and \ref{prop:unionNeg} we may find such an edge set $X_{\langle H', m' \rangle}$ that is maximal. 

We consider all such $\langle H', m' \rangle$ obtained by splitting off different sets of $p$ pairs of edges, and we take $\langle \overline{H}, \overline m \rangle$ so that $|E(\overline H)|$ is maximal. Let $X = E\langle \overline H, \overline m \rangle$.

The remainder of the proof exactly follows \cite{fekete}. We claim that there is a split edge $e = \{v, w; m_e\}$ in $\langle \overline H, \overline m \rangle$ which is not in $X$, which is shown using a basic combinatorial argument that we do not reproduce here. 

Fekete and Szeg\H{o} then use this fact to find a pair of edges which form an admissible splitting off, which contradicts the maximality of $p$.  \qed\\

\noindent {\it Proof of Theorem \ref{thm:16mod} (sketch).}\\
The ``if" part of this result is seen from the definition of gain-modified edge pinches. 

To see the other direction, we proceed by induction on the number of vertices. The \bbgood \ with only one vertex is $\langle P_3, m^*\rangle$, which we know is a \bbgood \ by definition. Let $\bbog$ be an arbitrary \bbgood \ with more than one vertex. Then it is not hard to show that there exists a node $s$ with degree at most 11 (since $H$ is $[6,3]$-tight). 

Using a simple combinatorial argument, we find that letting $n = 6 - deg_H(s)$, and $j= n-i_H(s)$ ($i_H(s)$ is the number of loops on $s$), we obtain $n$ and $j$ which satisfy the conditions of Theorem \ref{thm:22mod} (see \cite{fekete} for details). 

Theorem \ref{thm:22mod} demonstrates that $\bbog$ is obtained from a smaller body-bar orbit graph $\langle H', m' \rangle$ by a gain modified edge pinch $K^*(6, n, j)$. By induction, we know that  $\langle H', m' \rangle$ can be constructed from $\langle P_3, m^*\rangle$, and hence $\bbog$ can be too. \qed

\subsection{Proof of main result (Theorem \ref{thm:main})}
Theorem \ref{thm:dNecessary} establishes the necessity of the two conditions for generic rigidity. Suppose $\bbog$ satisfies conditions 1 and 2 of our main theorem. By Theorem \ref{thm:16mod}, it can be built up from a single vertex with three looks by a sequence of gain-modified edge pinches. By Theorem \ref{thm:bbPinches}, these gain-modified edge pinches preserve generic rigidity on $\Tor^3$, which completes the proof. \qed. 

%FURTHER WORK
\section{Further work}
As mentioned in the introduction, it was previously conjectured that the results of the present paper hold for $d$-dimensional body-bar orbit frameworks \cite{myThesis}. That is, we claim 
\begin{conj}
$\bbog$ is a generically minimally rigid body-bar framework on $\Tor^d$ if and only if 
\begin{enumerate}
	\item $|E(H)| = {d+1 \choose 2}|V(H)| - d$
	\item for all non-empty subsets $Y \subset E(H)$ of edges
		\begin{equation}
			|Y| \leq {d+1 \choose 2}|V(Y)| - {d+1 \choose 2} + \sum_{i=1}^{|\gs(Y)|}(d-i).
		\label{eq:gainSparseConj}
		\end{equation}	
\end{enumerate}
\label{conj:main}
\end{conj}

It is evident that the approach of the present paper does not immediately extend to this more general conjecture, since we have taken cases in the proofs of Propositions \ref{prop:union} and \ref{prop:unionNeg}, which are central to the proof of the main result. It is possible that some other method may prove useful in establishing the relevant results. Another possibility is that we abandon the approach of Fekete and Szeg\H{o}, and prove the result using an entirely different technique. In any case, we believe that the established result is a promising indication that the conjecture is true. 

It is possible to check that the expression (\ref{eq:gainSparseConj}) gives rise to a submodular function on the set of oriented edges labeled by elements of $\mathbb Z^d$. It then follows that the body-bar periodic orbit graphs $\bbog$ satisfying (\ref{eq:gainSparseConj}) form the independent sets of a matroid. We omit the details.

Other areas for further research may be the extension of these results to body-hinge frameworks, which provide yet another tool for studying molecular frameworks. The recent proof of the molecular conjecture \cite{MolecularConj} provides further motivation for this avenue of research. Finally, the development of algorithms, based on the pebble game algorithm \cite{pebbleGame} for checking whether a given body-bar orbit graph satisfies the sparsity condition of Theorem \ref{thm:main} would be of interest. 

\subsection{Tubes and slabs}
The paper of Walter Whiteley in this volume \cite{WWtubesSlabs} addresses the question of ``fragments" of periodic frameworks. That is, he considers finite pieces of periodic structures. As part of this set of questions he describes two types of fragments: {\it tubes} and {\it slabs}. Tubes are three-dimensional structures which are periodic in one direction. Slabs are three-dimensional structures which are periodic in two dimensions. 

From the result presented in this paper, it is possible to obtain extensions to tubes and slabs which are not fragments. Tubes can be seen to correspond to periodic frameworks $\bbog$ with $|\gs(E)|=1$, and slabs correspond to those frameworks with $|\gs(E)|=2$. Adapting the present results to these contexts is the subject of a forthcoming joint paper.  \\

\noindent {\it Acknowledgements.} The majority of this research was carried out at the Fields Institute for Research in Mathematical Sciences in Toronto, Canada. The author also wishes to thank Bernd Schulze and Walter Whiteley for their feedback.

%
%\bibliographystyle{abbrv} 
%\bibliography{royalSociety}

\end{document}